\documentclass[12pt]{amsart}
\usepackage{a4wide,enumerate,color,mathrsfs,graphicx,comment}
\allowdisplaybreaks

\let\pa\partial  
\let\na\nabla  
\let\eps\varepsilon  
\newcommand{\N}{{\mathbb N}}  
\newcommand{\R}{{\mathbb R}} 
\newcommand{\Z}{{\mathbb Z}}  
\newcommand{\blue}{\textcolor{black}}
\renewcommand{\L}{{\mathcal L}}

\newtheorem{theorem}{Theorem}   
\newtheorem{lemma}[theorem]{Lemma}   
\newtheorem{proposition}[theorem]{Proposition}   
\newtheorem{remark}[theorem]{Remark}   
\newtheorem{corollary}[theorem]{Corollary}  
  
\newtheorem{assumption}{Assumption} 

\begin{document}  

\title[Discrete Beckner inequalities]{
Discrete Beckner inequalities via the Bochner-Bakry-Emery approach for
Markov chains}

\author[A. J\"{u}ngel]{Ansgar J\"{u}ngel}
\address{Institute for Analysis and Scientific Computing, Vienna University of  
Technology, Wiedner Hauptstra\ss e 8--10, 1040 Wien, Austria}
\email{juengel@tuwien.ac.at}

\author[W. Yue]{Wen Yue}
\address{Institute for Analysis and Scientific Computing, Vienna University of  
Technology, Wiedner Hauptstra\ss e 8--10, 1040 Wien, Austria}
\email{wen.yue@tuwien.ac.at}

\date{\today}

\thanks{The authors acknowledge partial support from   
the Austrian Science Fund (FWF), grants P24304, P27352, and W1245}  

\begin{abstract}
Discrete convex Sobolev inequalities and Beckner inequalities are derived
for time-continuous Markov chains on finite state spaces. Beckner inequalities
interpolate between the modified logarithmic Sobolev inequality and 
the Poincar\'e inequality. Their proof is based
on the Bakry-Emery approach and on discrete Bochner-type inequalities
established by Caputo, Dai Pra, and Posta and recently extended by Fathi and Maas
for logarithmic entropies. The abstract result for convex entropies 
is applied to several Markov chains, including 
birth-death processes, zero-range processes, Bernoulli-Laplace models, and
random transposition models, and to a finite-volume discretization of a
one-dimensional Fokker-Planck equation, applying results by Mielke.
\end{abstract}



\maketitle


\section{Introduction}\label{sec.intro}

Convex Sobolev inequalities such as Poincar\'e and logarithmic Sobolev inequalities
play an important role in the analysis of the convergence to stationarity
for Markov processes. Besides implying exponential decay of the entropy, it is known
that these functional inequalities give useful concentration bounds \cite{BoTe06}
and hypercontractivity of the corresponding semigroup \cite{DiSa96},
and they are a natural tool to estimate mixing times \cite{MoTe06}.
There exists an extensive literature on the derivation of Poincar\'e
inequalities (or spectral gap estimates) and logarithmic Sobolev 
(or shorter: log-Sobolev) inequalities in the discrete and continuous setting;
see, e.g., 
the reviews \cite{DiSa96,GuZe03,MoTe06} and the books \cite{ABC00,BGL14,Wan05}.
An algorithm for the computation of the spectral gap is presented in \cite{ChSa14},
while corresponding estimates can be found in \cite{BCDP06,Che96,BuKe00}.
For log-Sobolev inequalities, we refer to \cite{BoLe98,CDP09,JSTV04}.

There are much less results on Beckner inequalities for Markov chains,
which interpolate between the Poincar\'e inequality and 
log-Sobolev inequality \cite{Bec89}. 
\blue{Such inequalities are of interest, for instance, in the large-time analysis of
Markov chains using general entropies or in numerical analysis, proving
the exponential decay of solutions to discretized partial differential
equations \cite{CJS16}.}
We are only aware of the paper by Bobkov and Tetali \cite{BoTe06},
where estimates on the constant of the Beckner inequality were derived for 
Bernoulli-Laplace and random transposition models. 
In this paper, we establish new bounds for discrete convex Sobolev
and Beckner inequalities for stochastic processes not studied in \cite{BoTe06}.

The technique of proof is the Bochner-Bakry-Emery method of Caputo et al.\
\cite{CDP09}, which was recently extended by Fathi and Maas in \cite{FaMa15} 
in the context of Ricci curvature bounds. 
The idea of the Bakry-Emery approach is to
relate the second time derivative of the entropy to its entropy production.
This relation is achieved by employing a discrete Bochner-type equation
which replaces the Bochner identity in the continuous case.

In order to make these ideas precise, consider a time-homogeneous Markov
process $(X_t)_{t\ge 0}$ with values in a finite state space $S$,
having an invariant measure $\pi$. We assume that the semigroup $(T_t)_{t\ge 0}$,
defined on $L^2(\pi)$ by $T_t f(x)=\textrm{E}[f(X_t):X_0=x]$, is strongly
right continuous, so that the infinitesimal generator $\L$ exists,
$T_t=e^{t\L}$. Given a probability measure $\mu$ on $S$, we denote by
$\mu T_t$ the distribution of $X_t$ assuming that $X_0$ is distributed according
to $\mu$. The rate of convergence of $\mu T_t$ to the invariant measure $\pi$
is a major topic in probability theory. It can be achieved by estimating
the time derivative of the relative entropy. 

Before explaining the entropy decay, we introduce some notation.
The relative entropy $h^\phi(\mu|\pi)$ of $\mu$ with respect to $\pi$ is defined by
$$
  h^\phi(\mu|\pi) = \pi\left[\phi\left(\frac{d\mu}{d\pi}\right)\right]
	= \sum_{\eta\in S}\pi(\eta)\phi\left(\frac{d\mu}{d\pi}\right)(\eta),
$$
where $\phi:\R_+\to\R_+$ is a smooth convex function such that 
$\phi(1)=0$ and $1/\phi''$ is concave, $\R_+=[0,\infty)$,
and $h^\phi(\mu|\pi)$ is meant to be infinite whenever $\mu\not\ll\pi$ 
or $\phi(d\mu/d\pi)\not\in 
L^1(\pi)$. The entropy can be defined on the set of probability densities
$f$ such that $\phi(f)\in L^1(\pi)$ by
$$
  \textrm{Ent}_\pi^\phi(f) = \pi[\phi(f)],
$$
so that $h^\phi(\mu|\pi) = \textrm{Ent}_\pi^\phi(d\mu/d\pi)$. 
When $\phi_1(s)=s(\log s-1)+1$, we obtain the logarithmic entropy 
and if $\phi_2(s)=s^2-2s+1$, $\textrm{Ent}_\pi^\phi(f)$ equals the variance of $f$, 
$\textrm{Var}_\pi(f)=\pi[f^2]-\pi[f]^2$. \blue{Another example is
$\phi_\alpha(s)=(s^\alpha-s)/(\alpha-1)-s+1$ for $1<\alpha\le 2$, which interpolates
between $\phi_1$ and $\phi_2$ in the sense that 
$\phi_\alpha(s)\to\phi_1(s)$ pointwise as $\alpha\to 1$ and 
$\phi_\alpha(s)=\phi_2$ if $\alpha=2$.}

Let $\rho_t=d(\mu T_t)/d\pi$ be the probability
density of the Markov chain at time $t\ge 0$. 
We assume in the following that the Markov chain
is reversible, i.e., the generator is self-adjoint in $L^2(\pi)$.
Then $\rho_t$ solves the Kolmogorov equation $\pa_t\rho_t=\L\rho_t$, $t>0$.
The idea of Bakry and Emery \cite{BaEm85} is to differentiate the entropy
twice with respect to time. A formal computation gives
\begin{equation}\label{1.dEntdt}
\begin{aligned}
  \frac{d}{dt}\textrm{Ent}_\pi^\phi(\rho_t) &= -{\mathcal E}(\phi'(\rho_t),\rho_t), \\
	\frac{d^2}{dt^2}\textrm{Ent}_\pi^\phi(\rho_t) 
	&= \pi[\L\phi'(\rho_t)\L\rho_t + \phi''(\rho_t)(\L\rho_t)^2\big],
\end{aligned}
\end{equation}
where ${\mathcal E}(f,g):=-\pi[f\L g]$ is the Dirichlet form of $\L$.
Now suppose that the following inequality holds for some $\lambda>0$:
\begin{equation}\label{1.ineq}
  \pi[\L\phi'(\rho)\L\rho + \phi''(\rho)(\L\rho)^2\big]
	\ge \lambda{\mathcal E}(\phi'(\rho),\rho), \quad t>0.
\end{equation}
This is equivalent to 
$\pa^2_t\textrm{Ent}_\pi^\phi(\rho)+\lambda\pa_t\textrm{Ent}_\pi^\phi(\rho)\ge 0$,
and by Gronwall's lemma, we conclude that $\pa_t\textrm{Ent}_\pi^\phi(\rho_t)$
converges to zero with exponential rate. Furthermore, 
integration over $(t,\infty)$ leads to
\begin{equation}\label{1.dHdt}
  \frac{d}{dt}\textrm{Ent}_\pi^\phi(\rho) 
	+ \lambda\textrm{Ent}_\pi^\phi(\rho) \le 0, \quad t>0,
\end{equation}
if we know that $\textrm{Ent}_\pi^\phi(\rho_t)\to 0$ as $t\to\infty$. 
On the one hand, this implies
exponential convergence of the relative entropy to zero,
i.e., $\textrm{Ent}_\pi^\phi(\rho_t)\le \textrm{Ent}_\pi^\phi(\rho_0)e^{-\lambda t}$.
On the other hand, \eqref{1.dHdt} is equivalent to the {\em convex Sobolev inequality}
\begin{equation}\label{1.gcsi}
  \lambda\textrm{Ent}_\pi^\phi(f)\le {\mathcal E}(\phi'(f),f),
\end{equation}
valid for all probability densities $f$. 

\blue{It is well known that if the so-called curvature-dimension condition 
$CD(\lambda,\infty)$ is satisfied, then the convex Sobolev inequality \eqref{1.gcsi}
is valid \cite[Section 1.16]{BGL14}. 
For instance, if $\L$ is the generator of the Ornstein-Uhlenbeck
process, $CD(\lambda,\infty)$ holds with $\lambda=1$ under the conditions that 
$\phi$ is convex and $1/\phi''$ is concave \cite{AMTU01}. In the discrete case,
the validity of \eqref{1.gcsi} is not known except in the logarithmic case
$\phi=\phi_1$. In this paper, we derive general conditions on $\phi$ 
that guarantee the validity of \eqref{1.gcsi}.}

For the special cases $\phi_1(s)=s(\log s-1)+1$ and $\phi_2(s)=s^2-2s+1$, we obtain
the {\em modified log-Sobolev inequality} and
{\em Poincar\'e inequality}, respectively,
\begin{equation}\label{1.csi}
  \lambda_M\textrm{Ent}_\pi^{\phi_1}(f) \le {\mathcal E}(\log f,f), \quad
	\lambda_P\textrm{Var}_\pi(f) \le {\mathcal E}(f,f).
\end{equation}
Note that if $\L$ is the generator of a reversible diffusion process,
we may write ${\mathcal E}(\log f,f)=4{\mathcal E}(f^{1/2},f^{1/2})$,
so the log-Sobolev inequality 
$\lambda_L\textrm{Ent}_\pi^{\phi_1}(f) \le {\mathcal E}(f^{1/2},f^{1/2})$
and the first inequality in \eqref{1.csi} coincide with $\lambda_M = 4\lambda_L$. 
This is generally
not true for Markov processes with jumps \cite{BoLe98}, but for reversible
processes, the relations $4\lambda_L\le\lambda_M\le 2\lambda_P$ hold
\cite{BoTe06,DiSa96}.

The aim of this paper is to determine conditions under which there
exists a constant $\lambda>0$ such that 
the (discrete) convex Sobolev inequality \eqref{1.gcsi}
and the exponential entropy decay
\begin{equation}\label{1.exp}
  \textrm{Ent}_\pi^\phi(\rho_t) \le e^{-\lambda t}\textrm{Ent}_\pi^\phi(f).
	\quad t>0,
\end{equation}
hold. Furthermore, we derive explicit constants $\lambda_B(\alpha)>0$
such that the (discrete) {\em Beckner inequality} holds:
\begin{equation}\label{1.beck}
  \lambda_B(\alpha)\textrm{Ent}_\pi^{\phi_\alpha}(\rho)
	\le \frac{\alpha}{\alpha-1}{\mathcal E}(\rho^{\alpha-1},\rho),
	\quad 1<\alpha\le 2.
\end{equation}

\blue{
The Beckner inequality is related to the modified log-Sobolev and Poincar\'e
inequalities. Indeed, if $\alpha\to 1$, \eqref{1.beck} becomes the
modified log-Sobolev inequality with 
$\lim_{\alpha\to 1}\lambda_B(\alpha)=\lambda_M$ and if $\alpha=2$,
\eqref{1.beck} equals the Poincar\'e inequality with $\lambda_B(2)=2\lambda_P$.
For $1<\alpha<2$, applying \eqref{1.beck} to functions of the form $1+\eps f$,
performing a Taylor expansion, and letting $\eps\to 0$ shows that
$\lambda_B(\alpha)\le 2\lambda_P$.}

According to the above discussion, inequalities \eqref{1.csi}-\eqref{1.beck}
are achieved by proving \eqref{1.ineq},
and the proof of this inequality is based on a discrete Bochner-type identity.
The idea to employ such an identity was first presented in
\cite{BCDP06}, elaborated later in \cite{CDP09,FaMa15}, and goes back to
\cite{Boc46}. The identity is obtained by identifying the Radon-Nikodym
derivative of a measure involving the jump rates of the Markov chain
\cite[Section 2]{BCDP06}. This allows one to relate terms with different
orders of ``discrete derivatives'' occuring in $\L$. For details, we refer to
Section \ref{sec.boch}. Our technique of proving \eqref{1.beck} is
completely different from the work \cite{BoTe06}, where an iteration method 
was used to derive discrete Beckner inequalities.

Fathi and Maas \cite{FaMa15} extended the results of Caputo et al. \cite{CDP09}.
The key idea of \cite{FaMa15} (and, by the way, of \cite{Mie13}) is the use of
the logarithmic mean 
$$
  \rho^*(\eta,\xi) = \frac{\rho(\eta)-\rho(\xi)}{\log\rho(\eta)-\log\rho(\xi)}
$$
in the analysis. The logarithmic mean allows for the discrete chain rule 
$\rho^*\na\log\rho = \na\rho$, where $\na\rho(\eta,\xi)=\rho(\eta)-\rho(\xi)$,
which naturally holds in the continuous case. This chain rule is needed to
treat the logarithmic entropy. In the case of general convex entropies, it is natural
to replace the logarithmic mean by 
\begin{equation}\label{1.mean}
  \widehat\rho(\eta,\xi) = \frac{\rho(\eta)-\rho(\xi)}{\phi'(\rho(\eta))
	-\phi'(\rho(\xi))}, \quad \phi\mbox{ convex},
\end{equation}
which satisfies the discrete chain rule $\widehat\rho\na\phi'(\rho)
=\na\rho$ since $\widehat\rho$ ``approximates'' $1/\phi''(\rho)$. 
When $\phi=\phi_\alpha$, we obtain the power mean
$$
  \widehat\rho(\eta,\xi) = \frac{\alpha-1}{\alpha}
	\frac{\rho(\eta)-\rho(\xi)}{\rho(\eta)^{\alpha-1}-\rho(\xi)^{\alpha-1}},
	\quad 1<\alpha<2.
$$
We remark that the idea to enforce a discrete chain rule is well known in
the design of structure-preserving numerical schemes and was used, e.g.,
in the construction of entropy-conservative finite-volume fluxes \cite{FMT12} and
in the discrete variational derivative method \cite{FuMa11}.

\blue{The novelty of this paper is the identification of the conditions on $\phi$ 
that are needed to apply the technique of \cite{CDP09,FaMa15}. 
It turns out that, besides convexity of $\phi$ and the concavity of $1/\phi''$, 
the concavity of 
\begin{equation}\label{1.theta}
  \theta(s,t) = \frac{s-t}{\phi'(s)-\phi'(t)}, \quad s\neq t,
	\quad \theta(s,s) = \frac{1}{\phi''(s)},
\end{equation}
is needed. 
This is not surprising since $\theta(s,t)$ is a discrete approximation
of $1/\phi''$, and the concavity of $1/\phi''$ is assumed in the continuous case. 
Conditions on $\phi$ that guarantee the concavity of $\theta$ are stated in
Lemma \ref{lem.phi1}.
Both the logarithmic mean and the power mean satisfy these conditions; 
see Lemma \ref{lem.theta1}.
The general theory can be applied to birth-death processes, thus yielding
new discrete convex Sobolev inequalities. For other stochastic processes
considered in this paper (zero-range processes, Bernoulli-Laplace models, 
random transposition models), a homogeneity property of $\theta$ is needed,
which restricts the class of admissible functions $\phi$. It turns out that
the logarithmic mean and the power mean satisfy this property; 
see Lemma \ref{lem.theta1}. For the 
mentioned processes, new discrete Beckner inequalities are derived.
}

The paper is organized as follows.
We detail the Bochner-Bakry-Emery method in Section \ref{sec.boch}.
The validity of the discrete Beckner inequality \eqref{1.beck} is reduced
to the validity of \blue{a modification of \eqref{1.ineq}}.
In Section \ref{sec.ex}, we apply the general technique to four stochastic
processes (as in \cite{FaMa15}): birth-death processes, zero-range processes,
Bernoulli-Laplace models, and random transposition models. Furthermore, 
the results for birth-death processes are applied to a finite-volume discretization
of a one-dimensional Fokker-Planck equation, yielding exponential decay of the
discrete entropy. The proof consists of a combination of the convex Sobolev inequality 
for birth-death processes and the results of Mielke \cite{Mie13}, who proved 
exponential decay for the logarithmic entropy. 

Our main conclusion is that the Bochner-Bakry-Emery approach is sufficiently 
flexible to be applicable to power functions and, in certain cases, to general
convex functions.


\section{The Bochner method}\label{sec.boch}

Let an irreducible and reversible Markov chain on a finite state space $S$
be given and let $\pi$ be the invariant measure. We write the generator $\L$
in the form
$$
  \L f(\eta) = \sum_{\gamma\in G}c(\eta,\gamma)\na_\gamma f(\eta),
$$
where $G$ is the set of allowed moves (represented by functions 
$\gamma:S\to S$), the mapping $c:S\times G\to[0,\infty)$
represents the jump rates, and $\na_\gamma f(\eta)=f(\gamma\eta)-f(\eta)$.
We observe that the generator of every finite Markov chain can be written
in this form. We assume the following two properties: For any 
$\gamma\in G$, there exists $\gamma^{-1}\in G$ satisfying $\gamma^{-1}\gamma\eta=\eta$
for all $\eta\in S$ with $c(\eta,\gamma)>0$. Furthermore, the reversibility
condition
$$
  \pi\bigg[\sum_{\gamma\in G}c(\eta,\gamma)F(\eta,\gamma)\bigg]
	= \pi\bigg[\sum_{\gamma\in G}c(\eta,\gamma)F(\gamma\eta,\gamma^{-1})\bigg]
$$
holds for all $F:S\times G\to\R$. Under reversibility, the Dirichlet form
can be written as
\begin{equation}\label{2.E}
  {\mathcal E}(f,g) = \frac12\pi\bigg[\sum_{\gamma\in G}c(\eta,\gamma)
	\na_\gamma f(\eta)\na_\gamma g(\eta)\bigg].
\end{equation}

For the discrete Bochner-type identity, we suppose as in \cite{CDP09}:

\begin{assumption}\label{ass}
There exists a function $R:S\times G\times G\to\R$ such that \\
{\rm (i)} $R(\eta,\gamma,\delta)=R(\eta,\delta,\gamma)$ for all $\eta\in S$, $\gamma$,
$\delta\in G$; \\
{\rm (ii)} for all bounded functions $\psi:S\times G\times G\to \R$,
$$
  \pi\bigg[\sum_{\gamma,\delta\in G}R(\eta,\gamma,\delta)
	\psi(\eta,\gamma,\delta)\bigg]
	= \pi\bigg[\sum_{\gamma,\delta\in G}R(\eta,\gamma,\delta)
	\psi(\gamma\eta,\gamma^{-1},\delta)\bigg].
$$
\noindent {\rm (iii)} $\gamma\delta\eta = \delta\gamma\eta$ for all $\eta\in S$,
$\gamma$, $\delta\in G$ with $R(\eta,\gamma,\delta)>0$.
\end{assumption}

The following lemma, which extends Lemma 2.3 in \cite{CDP09}, was proven in 
\cite[Lemma 3.3]{FaMa15}. It expresses a discrete Bochner-type identity.

\begin{lemma}\label{lem.14}
Let $\chi$, $\psi:S\to\R$ and let $\beta:S\times S\to\R$ be symmetric. Then
\begin{align*}
  \pi\bigg[\sum_{\gamma,\delta\in G} & R(\eta,\gamma,\delta)\beta(\eta,\delta\eta)
	\na_\delta\chi(\eta)\na_\gamma\psi(\eta)\bigg] \\
	&= \frac14\pi\bigg[\sum_{\gamma,\delta\in G}R(\eta,\gamma,\delta)
	\na_\gamma\big(\beta(\eta,\delta\eta)\na_\delta\chi(\eta)\big)\na_\delta
	\na_\gamma\psi(\eta)\bigg].
\end{align*}
\end{lemma}

The key estimate is contained in the following proposition that is an extension of
Theorem 3.5 in \cite{FaMa15} from the logarithmic case to the case of general
convex functions.

\begin{proposition}\label{prop}
\blue{Let $\phi\in C^3((0,\infty);(0,\infty))$ be convex such that $\phi(1)=0$,
$1/\phi''$ is concave on $(0,\infty)$, and let $\theta$, defined in 
\eqref{1.theta}, be concave.}
Assume that there exists a function $R$ satisfying Assumption \ref{ass} and define
$\Gamma(\eta,\gamma,\delta)=c(\eta,\gamma)c(\eta,\delta)-R(\eta,\gamma,\delta)$ for
$\eta\in S$ and $\gamma$, $\delta\in G$. Then, for any positive probability density
$\rho$,
\blue{\begin{align} 
  \pi\big[ & \L\phi'(\rho)\L\rho 
	+\phi'' (\rho) (\L\rho)^2\big] \label{3.prop} \\
	&\ge \pi\bigg[\sum_{\gamma,\delta\in G}\Gamma(\eta,\gamma,\delta)\Big(
	\na_\gamma\phi'(\rho(\eta))\na_\delta\rho(\eta)
	+ \phi''(\rho(\eta))\na_\gamma\rho(\eta)
	\na_\delta\rho(\eta)\Big)\bigg]. \nonumber
\end{align}}
\end{proposition}

\begin{remark}\label{rem.hat}\rm
In Lemma \ref{lem.phi1} (see Appendix), conditions on $\phi$ are 
stated guaranteeing the concavity of $\theta$.
We introduce the following notation:
\begin{align}
  \widehat\rho(\eta,\delta\eta) &= \theta(\rho(\eta),\rho(\delta\eta))
	= \frac{\rho(\delta\eta)-\rho(\eta)}{\phi'(\rho(\delta\eta))
	-\phi'(\rho(\eta))}
	= \frac{\na_\delta\rho(\eta)}{\na_\delta\phi'(\rho(\eta))}, \label{3.hatrho} \\
	\widehat\rho_1(\eta,\delta\eta) &= \pa_1\theta(\rho(\eta),\rho(\delta\eta))
	= -\frac{1}{\na_\delta\phi'(\rho(\eta))}
	+\frac{\na_\delta\rho(\eta)\phi''(\rho(\eta))}{(\na_\delta\phi'(\rho(\eta)))^2}, 
	\label{3.hatrho1} \\
	\widehat\rho_2(\eta,\delta\eta) &= \pa_2\theta(\rho(\eta),\rho(\delta\eta))
	= \widehat\rho_1(\delta\eta,\eta),
	\label{3.hatrho2}
\end{align}
where $\pa_1\theta$ and $\pa_2\theta$ are the partial derivatives of $\theta$
with respect to the first and second variable, respectively. 
\qed
\end{remark}

\begin{proof}[Proof of Proposition \ref{prop}.]
The first term on the left-hand side of \eqref{3.prop} can be written as
follows, using the definitions of $\L$, $\widehat\rho$, and $\Gamma$:
\begin{align*}
  \pi\big[\L\phi'(\rho) \L\rho\big]
	&= \pi\bigg[\sum_{\gamma,\delta\in G}c(\eta,\gamma)c(\eta,\delta)\na_\gamma
	\phi' (\rho) \na_\delta\rho(\eta)\bigg] \\
	&= \pi\bigg[\sum_{\gamma,\delta\in G}c(\eta,\gamma)c(\eta,\delta)
	\widehat\rho(\eta,\delta\eta)\na_\gamma
	\phi'(\rho(\eta))\na_\delta\phi'(\rho(\eta))\bigg] \\
	&= \pi\bigg[\sum_{\gamma,\delta\in G}R(\eta,\gamma,\delta)
	\widehat\rho(\eta,\delta\eta)\na_\gamma
	\phi'(\rho(\eta))\na_\delta\phi'(\rho(\eta))\bigg] \\
	&\phantom{xx}{}+ \pi\bigg[\sum_{\gamma,\delta\in G}\Gamma(\eta,\gamma,\delta)
	\widehat\rho(\eta,\delta\eta)\na_\gamma
	\phi'(\rho(\eta))\na_\delta\phi'(\rho(\eta))\bigg].
\end{align*}
By Lemma \ref{lem.14} with $\beta(\eta,\delta\eta)=\widehat\rho(\eta,\delta\eta)$,
the first term on the right-hand side of the previous equation
can be rewritten, leading to $\pi[\L\phi'(\rho)\L\rho]=A_1+A_2$, where
\begin{align*}
  A_1 &= \frac14\pi\bigg[\sum_{\gamma,\delta\in G}R(\eta,\gamma,\delta)
	\na_\gamma\big(\widehat\rho(\eta,\delta\eta)\na_\delta\phi'(\rho(\eta))\big)
	\na_\delta\na_\gamma\phi'(\rho(\eta))\bigg], \\
	A_2 &= \pi\bigg[\sum_{\gamma,\delta\in G}\Gamma(\eta,\gamma,\delta)
	\widehat\rho(\eta,\delta\eta)\na_\gamma
	\phi'(\rho(\eta))\na_\delta\phi'(\rho(\eta))\bigg].
\end{align*}

Next, we reformulate the second term on the left-hand side of \eqref{3.prop},
using the definitions of $\L$, $\widehat\rho_1$, and $\Gamma$:
\begin{align*}
\pi\big[\phi''(\rho)(\L\rho)^2\big]
	&= \pi\bigg[\sum_{\gamma,\delta\in G}c(\eta,\gamma)c(\eta,\delta)
	\na_\gamma\rho(\eta)\na_\delta\rho(\eta)\phi''(\rho(\eta))\bigg] \\
	&= \pi\bigg[\sum_{\gamma,\delta\in G}c(\eta,\gamma)c(\eta,\delta) 
	\na_\gamma\rho(\eta)	\widehat\rho_1(\eta,\delta\eta)
	(\na_\delta\phi'(\rho(\eta)))^2\bigg] \\
	&\phantom{xx}{}+ \pi\bigg[\sum_{\gamma,\delta\in G}c(\eta,\gamma)c(\eta,\delta)
	\na_\gamma\rho(\eta)\na_\delta\phi'(\rho(\eta))\bigg] \\
  &= \pi\bigg[\sum_{\gamma,\delta\in G}R(\eta,\gamma,\delta)\na_\gamma\rho(\eta)
	\widehat\rho_1(\eta,\delta\eta)(\na_\delta\phi'(\rho(\eta)))^2\bigg] \\
	&\phantom{xx}{}+ \pi\bigg[\sum_{\gamma,\delta\in G}\Gamma(\eta,\gamma,\delta)
	\na_\gamma\rho(\eta)\widehat\rho_1(\eta,\delta\eta)
	(\na_\delta\phi'(\rho(\eta)))^2\bigg] \\
	&\phantom{xx}{}+  \pi\bigg[\sum_{\gamma,\delta\in G}c(\eta,\gamma)c(\eta,\delta)
	\na_\gamma\rho(\eta)\na_\delta\phi'(\rho(\eta))\bigg] \\
	&=: B_1 + B_2 + (A_1+A_2).
\end{align*}
Then the left-hand side of \eqref{3.prop} is given by
$$
\pi\big[ \L\phi'(\rho)\L\rho 
	+\phi'' (\rho) (\L\rho)^2\big]
	= (B_1+2A_1) + (B_2+2A_2),
$$
and we will estimate $B_1+2A_1$ and $B_2+2A_2$ separately.

First, we treat $B_2+2A_2$. Inserting the definition of 
$\widehat\rho(\eta,\delta\eta)$ and rearranging the terms, we find that
\begin{align*}
  B_2+2A_2 &=  \pi\bigg[\sum_{\gamma,\delta\in G}\Gamma(\eta,\gamma,\delta)
	\widehat\rho_1(\eta,\delta\eta)\na_\gamma\rho(\eta)
	(\na_\delta\phi'(\rho(\eta)))^2\bigg] \\
	&\phantom{xx}{}+ 2\pi\bigg[\sum_{\gamma,\delta\in G}\Gamma(\eta,\gamma,\delta)
	\widehat\rho(\eta,\delta\eta)\na_\gamma
	\phi'(\rho(\eta))\na_\delta\phi'(\rho(\eta))\bigg] \\
	&= \pi\bigg[\sum_{\gamma,\delta\in G}\Gamma(\eta,\gamma,\delta)\na_\gamma
	\phi'(\rho(\eta))\na_\delta\rho(\eta)\bigg] \\
	&\phantom{xx}{}+ \pi\bigg[\sum_{\gamma,\delta\in G}
	\Gamma(\eta,\gamma,\delta)\na_\gamma\rho(\eta)\na_\delta\rho(\eta)
	\phi''(\rho(\eta))\bigg],
\end{align*}
which is exactly the right-hand side of \eqref{3.prop}. 
Thus, it remains to prove that $B_1+2A_1\ge 0$. 

To this end, we reformulate
$B_1$, employing Assumption \ref{ass} (i)-(ii) and identity \eqref{3.hatrho2}:
\begin{align}
  B_1 &= \pi\bigg[\sum_{\gamma,\delta\in G}R(\eta,\gamma,\delta)
	\widehat\rho_1(\eta,\delta\eta)\na_\gamma\rho(\eta)
	(\na_{\delta}\phi'(\rho(\eta)))^2\bigg] \label{3.aux11} \\
	&= \pi\bigg[\sum_{\gamma,\delta\in G}R(\eta,\gamma,\delta)
	\widehat\rho_1(\delta\eta,\eta)\na_\gamma\rho(\delta\eta)
	(\na_{\delta^{-1}}\phi'(\rho(\delta\eta)))^2\bigg] \nonumber \\
  &= \pi\bigg[\sum_{\gamma,\delta\in G}R(\eta,\gamma,\delta)
	\widehat\rho_2(\eta,\delta\eta)\na_\gamma\rho(\delta\eta)
	(\na_{\delta}\phi'(\rho(\eta)))^2\bigg], \label{3.aux12}
\end{align}
since $\na_{\delta^{-1}}\phi'(\rho(\delta\eta)) 
= -\na_\delta\phi'(\rho(\eta))$. 
Averaging \eqref{3.aux11} and \eqref{3.aux12} gives
$$
  B_1 = \frac12\pi\bigg[\sum_{\gamma,\delta\in G}R(\eta,\gamma,\delta)
	\Big(\widehat\rho_1(\eta,\delta\eta)\na_\gamma\rho(\eta)
	+ \widehat\rho_2(\eta,\delta\eta)\na_\gamma\rho(\delta\eta)\Big)
	(\na_{\delta}\phi'(\rho(\eta)))^2\bigg].
$$
By \eqref{uvst} from Lemma \ref{lem.phi1} (see Appendix)
with $u=\rho(\gamma\eta)$, $v=\rho(\gamma\delta\eta)$, 
$s=\rho(\eta)$, and $t=\rho(\delta\eta)$, it follows that
$$
  \widehat\rho_1(\eta,\delta\eta)\na_\gamma\rho(\eta) 
	+ \widehat\rho_2(\eta,\delta\eta)\na_\gamma\rho(\delta\eta)
  \ge \na_\gamma\widehat\rho(\eta,\delta\eta),
$$
and we infer from the definition of $A_1$ that
\begin{align}
  B_1 + 2A_1 &\ge \frac12\pi\bigg[\sum_{\gamma,\delta\in G}R(\eta,\gamma,\delta)
	\Big\{\na_\gamma\widehat\rho(\eta,\delta\eta)(\na_{\delta}\phi'(\rho(\eta)))^2 
	\label{3.B1A1} \\
	&\phantom{xx}{}
	+	\na_\gamma\big(\widehat\rho(\eta,\delta\eta)\na_\delta\phi'(\rho(\eta))\big)
	\na_\delta\na_\gamma \phi'(\rho(\eta))\Big\}\bigg]. \nonumber
\end{align}

The following identity has been used in the proof of Theorem 3.5
in \cite{FaMa15}:
\begin{align}
  \na_\gamma \widehat\rho (\eta,\delta\eta) & (\na_\delta\psi(\eta))^2
  + \na_\gamma\big(\widehat\rho(\eta,\delta\eta)\na_\delta\psi(\eta)\big)
	\na_\delta\na_\gamma\psi(\eta) \label{3.id} \\
	&= \widehat\rho(\gamma\eta,\gamma\delta\eta)(\na_\gamma\na_\delta\psi(\eta))^2
	- \widehat\rho(\eta,\delta\eta)\na_\delta\psi(\gamma\eta)\na_\delta\psi(\eta)
	\nonumber \\
	&\phantom{xx}{}
	+ \widehat\rho(\gamma\eta,\delta\gamma\eta)\na_\delta\psi(\gamma\eta)
	\na_\delta\psi(\eta). \nonumber
\end{align}
It can be verified by elementary computations. Taking 
$\psi(\eta)= \phi'(\rho(\eta))$, the left-hand side of \eqref{3.id} equals
the expression in the curly brackets of \eqref{3.B1A1}, and we conclude that
\begin{align*}
  B_1+2A_1 &\ge \frac12\pi\bigg[\sum_{\gamma,\delta\in G}R(\eta,\gamma,\delta)
	\widehat\rho(\gamma\eta,\gamma\delta\eta)(\na_\gamma\na_\delta
	\phi'(\rho(\eta)))^2\bigg] \\
	&\phantom{xx}{}- \frac12\pi\bigg[\sum_{\gamma,\delta\in G}R(\eta,\gamma,\delta)
	\widehat\rho(\eta,\delta\eta)\na_\delta
	\phi'(\rho(\gamma\eta))\na_\delta\phi'(\rho(\eta))\bigg] \\
	&\phantom{xx}{}+ \frac12\pi\bigg[\sum_{\gamma,\delta\in G}R(\eta,\gamma,\delta)
	\widehat\rho(\gamma\eta,\delta\gamma\eta)\na_\delta\phi'(\rho(\gamma\eta))
	\na_\delta\phi'(\rho(\eta))\bigg].
\end{align*}
It follows from Assumption \ref{ass} (ii)-(iii) that the second and third term on the
right-hand side cancel. The first term being nonnegative, we infer that
$B_1+2A_1\ge 0$, which concludes the proof.
\end{proof}

The following corollary is a consequence of Proposition \ref{prop}.

\begin{corollary}\label{coro}
\blue{Let $\phi\in C^3((0,\infty);(0,\infty))$ be convex such that $\phi(1)=0$,
$1/\phi''$ is concave on $(0,\infty)$, and let $\theta$, defined in 
\eqref{1.theta}, be concave.}
Suppose that there exists a constant $\lambda>0$ such that for all positive
probability densities $\rho$,
\blue{\begin{align}
  \pi\bigg[ & \sum_{\gamma,\delta\in G}\Gamma(\eta,\gamma,\delta)\Big(
	\na_\gamma\phi'(\rho(\eta))\na_\delta\rho(\eta) 
	+ \phi''(\rho(\eta))\na_\gamma\rho(\eta)
	\na_\delta\rho(\eta)\Big)\bigg] \label{ineq} \\
	&\ge \frac{\lambda}{2}\pi\bigg[\sum_{\gamma\in G}c(\eta,\gamma)
	\na_\gamma\phi'(\rho(\eta))\na_\gamma\rho(\eta)\bigg]. \nonumber
\end{align}}
Then the convex Sobolev inequality \eqref{1.gcsi}, the decay of the Dirichlet form
\begin{equation}\label{3.decay}
  {\mathcal E}(\phi'(e^{t\L}\rho),e^{t\L}\rho)
	\le e^{-\lambda t}{\mathcal E}(\phi'(\rho),\rho), \quad t>0,
\end{equation}
and the decay of the entropy \eqref{1.exp}
hold for all positive probability densities $\rho$.
\end{corollary}

\begin{proof}
By Proposition \ref{prop} and representation \eqref{2.E} of the
Dirichlet form, it follows from \eqref{ineq} that
$$
  \pi[\L\phi'(\rho)\L\rho] + \pi[(\L\rho)^2\phi''(\rho)]
	\ge \lambda{\mathcal E}(\phi'(\rho),\rho).
$$
Taking into account \eqref{1.dEntdt}, this inequality is equivalent to
\begin{equation}\label{2.dd}
  \frac{d^2}{dt^2}\textrm{Ent}_\pi^\phi(\rho_t) 
	\ge -\lambda\frac{d}{dt}\textrm{Ent}_\pi^\phi(\rho_t).
\end{equation}
Using Gronwall's lemma, we infer that 
$0=\lim_{t\to\infty}(-\pa_t\textrm{Ent}_\pi^\phi(\rho_t))$.
Furthermore, as $\pi$ is an invariant measure, $\rho_t\to 1$ and
$\textrm{Ent}_\pi(\rho_t)\to 0$ as $t\to\infty$. 
Therefore, integrating \eqref{2.dd} over
$(0,\infty)$, we conclude that
$$
  -{\mathcal E}(\phi'(\rho_0),\rho_0)
	= \frac{d}{dt}\textrm{Ent}_\pi^\phi(\rho_0) 
	\le -\lambda\textrm{Ent}_\pi^\phi(\rho_0),
$$
and this is exactly the convex Sobolev inequality \eqref{1.gcsi}.
\end{proof}



\section{Examples}\label{sec.ex}

In this section, we consider some stochastic processes analyzed in
\cite{CDP09,FaMa15} but for logarithmic entropies only. 
\blue{For birth-death processes, we are able to allow for general convex
entropies, while for the remaining cases (zero-range processes,
Bernoulli-Laplace models, Random transposition models), only power entropies
with $\phi=\phi_\alpha$ can be considered. The reason is that we need additional
features of $\phi$ that seem to be satisfied only under certain homogeneity
properties. These features are summarized in Lemma \ref{lem.theta1}.} 
Our notation follows that of \cite{CDP09}.

\subsection{Birth-death processes}\label{sec.bdp}

We investigate birth-death processes on $\N=\{0,1,2,\ldots\}$ with generator
$$
  \L f(n) = a(n)\na_+ f(n) + b(n)\na_- f(n), \quad n\in\N,
$$
where $a$ and $b$ are nonnegative functions on $\N$ satisfying $b(0)=0$.
The function $a$ represents the rate of birth, the function $b$ the rate of death.
The set of allowed moves is given by $G=\{+,-\}$, where $+(n)=n+1$ for $n\in\N$ 
and $-(n)=n-1$ for $n\ge 1$, $-(0)=0$. In particular,
$\na_\pm f(n)=f(n\pm 1)-f(n)$. According to the notation of Section \ref{sec.boch},
$c(n,+)=a(n)$ and $c(n,-)=b(n)$. 

\blue{Since we considered in the previous section finite state spaces,
we need to assume that the transition rates $a(n)$ and $b(n)$ vanish
for sufficiently large values of $n$ in order to fit into this framework. 
Another possibility is to consider
finitely supported test functions. According to \cite{Maa16}, this
case may be covered by using the results of Daniri and Savar\'e \cite{DaSa08}.
We expect that the result below still holds for countable Markov chains, 
but we leave the proof for future works; also see \cite[Remark 4.2]{FaMa15}.}

We suppose that this Markov chain is irreducible and reversible, i.e., there
exists a probability measure $\pi$ on $\N$ satisfying the detailed-balance condition
\begin{equation}\label{41.db}
  a(n)\pi(n) = b(n+1)\pi(n+1), \quad n\in\N.
\end{equation}
The following theorem is a consequence of Corollary \ref{coro}, applied
to birth-death processes.

\begin{theorem}\label{thm.bdp}
Let $\lambda>0$ \blue{and let $\phi$ satisfy the assumptions stated in
Proposition \ref{prop}.}
Assume that $a$ is nonincreasing, $b$ is nondecreasing, and
\begin{equation}\label{41.cond}
  a(n)-a(n+1) + b(n+1)-b(n) + \Theta\big(a(n)-a(n+1),b(n+1)-b(n)\big) \ge \lambda
\end{equation}
for all $n\in\N$, where
$$
  \Theta(A,B) := \inf_{s,t>0}\theta(s,t)(A\phi''(s)+B\phi''(t)), 
	\quad A,B\ge 0,
$$
and $\theta(s,t)=(s-t)/(\phi'(s)-\phi'(t))$ for $s\neq t$. 
Then the convex Sobolev inequality \eqref{1.csi} and the decay estimates 
\eqref{1.exp} and \eqref{3.decay} hold with constant $\lambda$.
\end{theorem}

The mapping $\Theta$ generalizes the function in \cite[Section 4.1]{FaMa15}.
For the special case $\phi(s)=\phi_\alpha(s)=(s^\alpha-s)/(\alpha-1)-s+1$,
Lemma \ref{lem.theta3} in the Appendix shows that 
$\Theta(A,B)\ge (\alpha-1)(A+B)$. Moreover, $\Theta(A,B)=A+B$ if $\alpha=2$.
Figure \ref{fig.Theta} illustrates the ``sharpness'' of the inequality
$\Theta(A,B)\ge (\alpha-1)(A+B)$ for $\alpha$ close to one. 

\begin{figure}[ht]
\centering
\includegraphics[width=70mm]{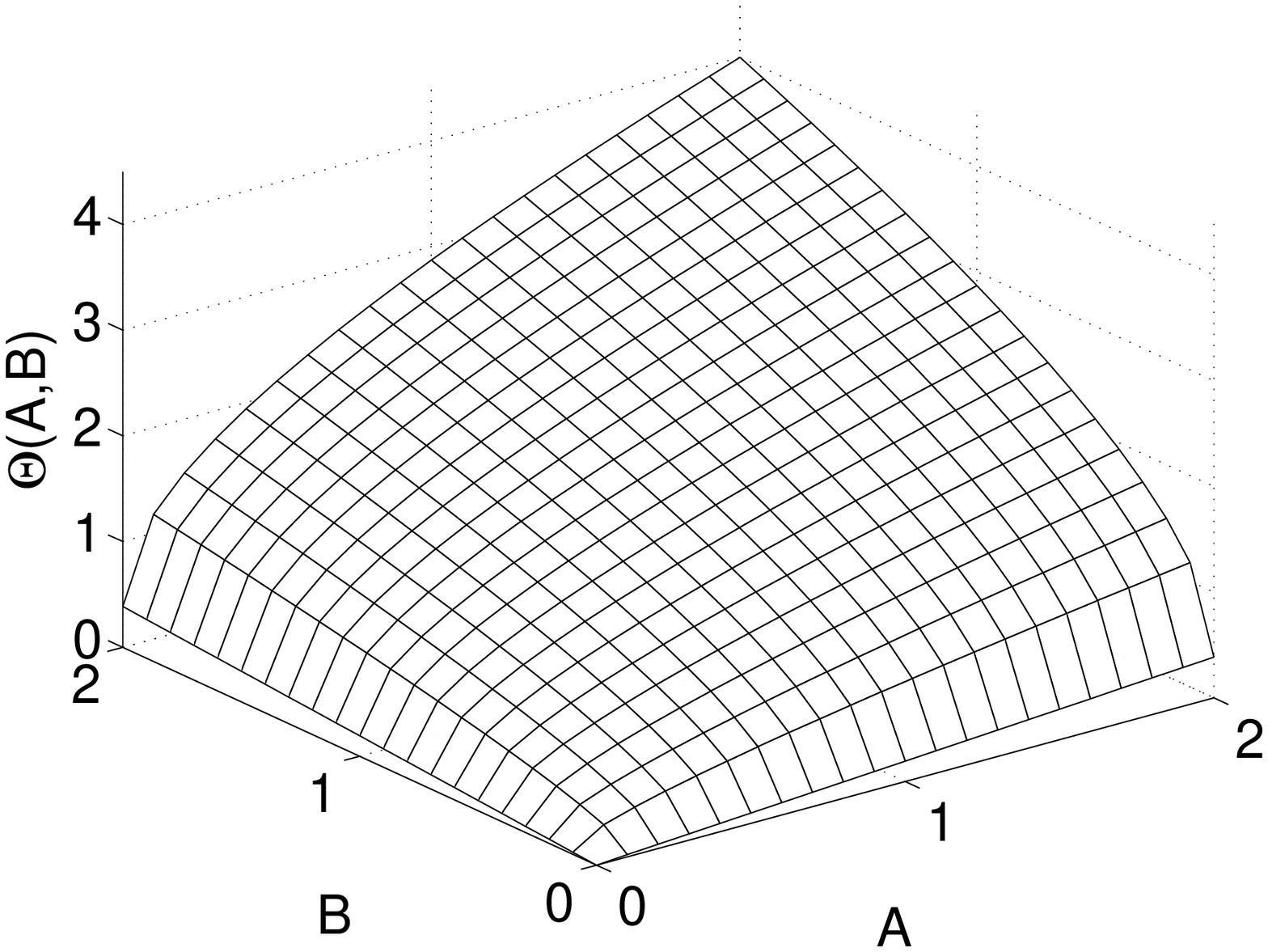}\quad
\includegraphics[width=70mm]{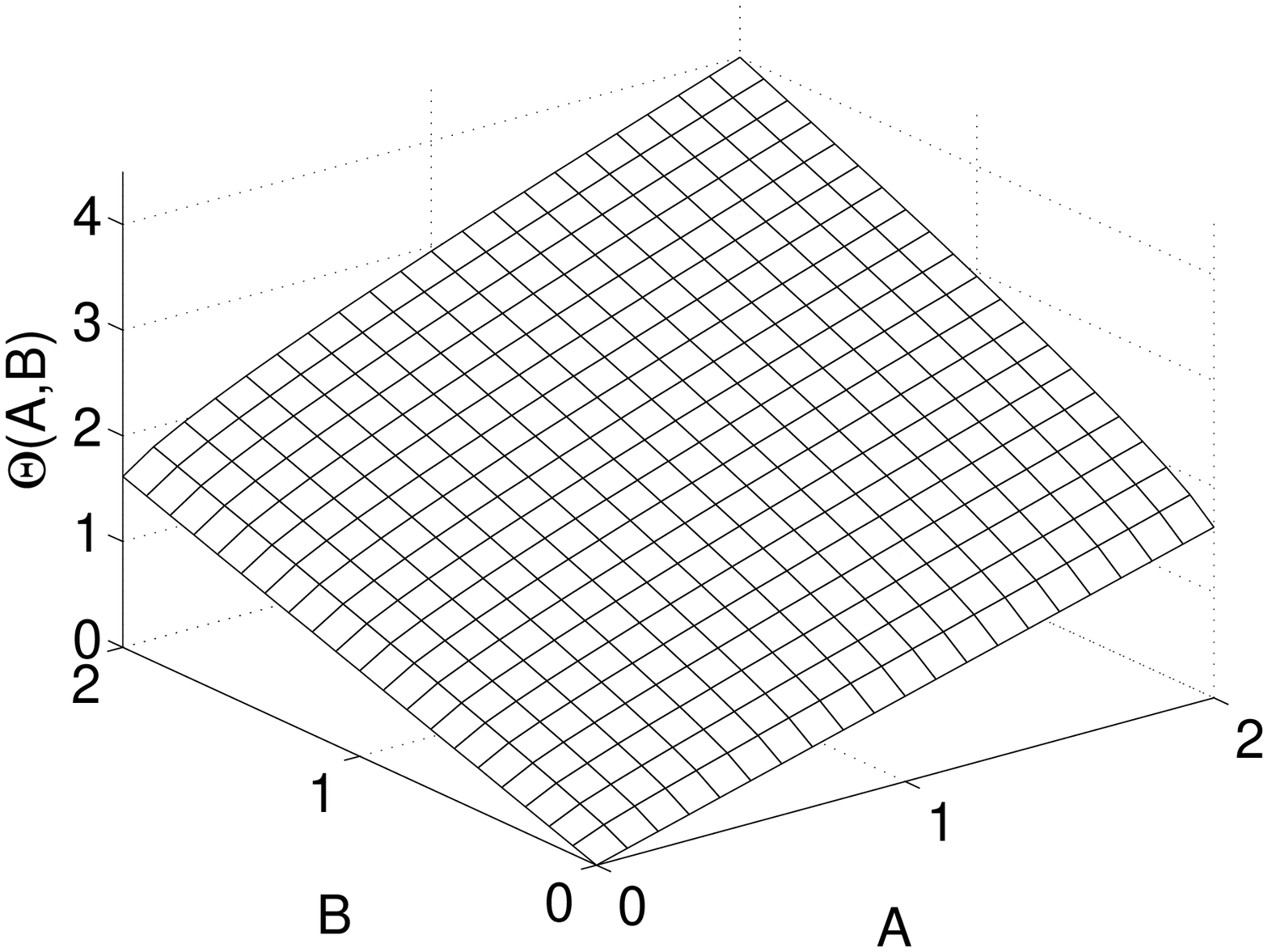}
\caption{Illustration of $\Theta(A,B)$, defined in \eqref{1.theta}, for 
$\alpha=1.01$ (left) and $\alpha=1.8$ (right).}
\label{fig.Theta}
\end{figure}

\begin{remark}\rm
Estimates for Poincar\'e inequalities for Markov chains are given in, e.g., 
\cite{Che96,Che03,Mic99}. 
The same criterion as in \eqref{41.cond} was obtained 
in \cite[Theorem 5.1]{Mie13} and \cite[Theorem 4.1]{FaMa15} for the logarithmic
entropy ($\alpha\to 1$). From Lemma \ref{lem.theta3} we conclude 
that the Beckner constant can be estimated by $\lambda\ge\alpha
(a(n)-a(n+1)+b(n-1)-b(n))$.
There exist sufficient and necessary conditions on $\pi$ and $a(n)$ such that
an interpolation between the Poincar\'e and log-Sobolev inequality 
holds, but without estimates on the constant \cite[Theorem 6.2.4]{Wan05}. 
\qed
\end{remark}

\begin{proof}
We define as in \cite[Section 3]{CDP09}
\begin{align*}
  R(n,+,+) &= a(n)a(n+1), \quad R(n,-,-) = b(n)b(n-1), \\
	R(n,+,-) &= R(n,-,+) = a(n)b(n).
\end{align*}
This function satisfies Assumption \ref{ass}. In particular, (ii) follows from the 
detailed-balance condition \eqref{41.db}. As before, we set
$\Gamma(n,\gamma,\delta) = c(n,\gamma)c(n,\delta) - R(n,\gamma,\delta)$ for
$\gamma$, $\delta\in G$. According to Corollary \ref{coro}, we only need to
verify \eqref{ineq}. The left-hand side equals
\begin{align*}
  \pi\bigg[ & \sum_{\gamma,\delta\in G}\Gamma(n,\gamma,\delta)\Big(
	\na_\gamma\phi'(\rho(n))\na_\delta\rho(n) + \na_\gamma\rho(n)
	\na_\delta\rho(n)\phi''(\rho(n))\Big)\bigg] \\
	&= \pi\Big[a(n)(a(n)-a(n+1))\Big(\na_+\phi'(\rho(n))\na_+\rho(n)
	+ (\na_+\rho(n))^2\phi''(\rho(n))\Big)\Big] \\
	&\phantom{xx}{}
	+ \pi\Big[b(n)(b(n)-b(n-1))\Big(\na_-\phi'(\rho(n))\na_-\rho(n)
	+ (\na_-\rho(n))^2\phi''(\rho(n))\Big)\Big],
\end{align*}
since the sum over all $\gamma$, $\delta\in G$ consists of four terms
$(+,+)$, $(-,-)$, $(+,-)$, and $(-,+)$, and because of 
$\Gamma(n,+,-)=\Gamma(n,-,+)=0$, only two terms do not vanish.
Now, we perform the change $n\mapsto n+1$ in the second term and replace
$\pi(n+1)b(n+1)$ by $\pi(n)a(n)$, according to the detailed-balance 
condition \eqref{41.db}.
Observing that $b(0)=0$ and $\na_-\rho(n+1)=-\na_+\rho(n)$, this leads to
\begin{align}
  \pi\bigg[ & \sum_{\gamma,\delta\in G}\Gamma(n,\gamma,\delta)\Big(
	\na_\gamma\phi'(\rho(n))\na_\delta\rho(n) + \na_\gamma\rho(n)
	\na_\delta\rho(n)\phi''(\rho(n))\Big)\bigg] \nonumber \\
	&= \pi\Big[a(n)(a(n)-a(n+1))\Big(\na_+\phi'(\rho(n))\na_+\rho(n)
	+ (\na_+\rho(n))^2\phi''(\rho(n))\Big)\Big] \nonumber \\
	&\phantom{xx}{}
	+ \pi\Big[a(n)(b(n+1)-b(n))\Big(\na_+\phi'(\rho(n))\na_+\rho(n)
	+ (\na_+\rho(n))^2\phi''(\rho(n+1))\Big)\Big] \nonumber \\
	&= \pi\Big[a(n)\Big(a(n)-a(n+1)+b(n+1)-b(n)\Big)\na_+\phi'(\rho(n))
	\na_+\rho(n)\Big] \nonumber \\
	&\phantom{xx}{}
	+ \pi\Big[a(n)\Big((a(n)-a(n+1))\phi''(\rho(n))
	+ (b(n+1)-b(n))\phi''(\rho(n+1))\Big) \nonumber \\
	&\phantom{xx}{}\times\widehat\rho(n,n+1)\na_+\phi'(\rho(n))
	\na_+\rho(n)\Big] \nonumber \\
  &\ge \lambda\pi\big[a(n)\na_+\phi'(\rho(n))\na_+\rho(n)\big], \nonumber
\end{align}
where in the last step we employed \eqref{41.cond}.
Using again the detailed-balance condition \eqref{41.db} and the identity
$\na_-\rho(n)=-\na_+\rho(n-1)$, the right-hand side
of \eqref{ineq} becomes
\begin{align*}
  \frac{\lambda}{2}\pi\bigg[ & \sum_{\gamma\in G}c(n,\gamma)\na_\gamma
	\phi'(\rho(n))\na_\gamma\rho(n)\bigg] \\
	&= \frac{\lambda}{2}\pi\Big[a(n)\na_+\phi'(\rho(n))\na_+\rho(n)\Big]
	+ \frac{\lambda}{2}\pi\Big[b(n)\na_-\phi'(\rho(n))\na_-\rho(n)\Big] \\
	&= \frac{\lambda}{2}\pi\Big[a(n)\na_+\phi'(\rho(n))\na_+\rho(n)\Big]
	+ \frac{\lambda}{2}\pi\Big[a(n)\na_+\phi'(\rho(n))\na_+\rho(n)\Big] \\
	&= \lambda\pi\big[a(n)\na_+\phi'(\rho(n))\na_+\rho(n)\big].
\end{align*}
Combining the above computations, inequality \eqref{ineq} follows.
\end{proof}


\subsection{Zero-range processes}\label{sec.zrp}

A zero-range process describes a stochastically interacting particle system that
may exhibit phase separation; see, e.g., \cite{MCG12}. The system consists of
finitely many particles moving in a finite number of sites $\{1,2,\ldots,L\}$. 
We adopt the notation of \cite{CDP09}. Let $\eta_x\in\N$ denote the number of
particles at $x\in\{1,2,\ldots,L\}$. Then the state space is $S=\N^L$. 
The configuration is changed by moving a particle from an (occupied) site $x$
to another site $y$. Correspondingly, the set $G$ of allowed moves is given
by self-mappings of $S$ which are of the form $\eta\mapsto\eta^{xy}$, where
$x$, $y\in\{1,2,\ldots,L\}$, $x\neq y$, and
$$
  \eta^{xy}_z = \left\{\begin{array}{ll}
	\eta_z & \mbox{if }z\notin\{x,y\}\mbox{ or }\eta_x=0, \\
	\eta_z-1 & \mbox{for }z=x\mbox{ and }\eta_x>0, \\
	\eta_z+1 & \mbox{for }z=y\mbox{ and }\eta_x>0.
	\end{array}\right.
$$
We denote by $xy$ the mapping $\eta\mapsto\eta^{xy}$ (such that $xy(\eta)=\eta^{xy}$)
and set $\na_{xy}f(\eta)=f(\eta^{xy})-f(\eta)$ for $\eta\in S$. 

The jump rates are functions $c_x:\N\to\R_+$ for $x\in\{1,2,\ldots,L\}$
satisyfing $c_x(0)=0$ and $c_x(n)>0$ for $n>0$. They describe the rate at which
a particle is moved from site $x$ to site $y$, with randomly chosen $y$, 
with uniform probability on $\{1,2,\ldots,L\}$. Then the rate 
$c(\eta,xy)$ for moving a particle from $x$ to $y$ is $c_x(\eta_x)/L$, 
and the generator of the Markov chain becomes
$$
  \L f(\eta) = \frac{1}{L}\sum_{x,y}c_x(\eta_x)\na_{xy}f(\eta),
$$
where the sum extends to all $x$, $y\in\{1,2,\ldots,L\}$. 
The number of particles $N=\sum_{1\le x\le L}\eta_x$ is conserved. We define
the probability measure $\pi_N$ on configurations with $N$ particles by
$$
  \pi_N(\eta) = \frac{1}{Z_N}\prod_{x=1}^L\prod_{k=1}^{\eta_x}\frac{1}{c_x(k)},
$$
where $Z_N>0$ the (finite) normalization constant. Since
\begin{equation}\label{42.db}
  \pi[c_x(\eta_x)g(\eta)] = \pi[c_y(\eta_y)g(\eta^{yx})]
\end{equation}
holds for all functions
$g:S\to\R$, the Markov chain is reversible with respect to $\pi_N$. In the
following, we fix the number of particles $N$ and omit the subscript $N$. 

\begin{theorem}\label{thm.zrp}
\blue{Let $\phi(s)=(s^\alpha-s)/(\alpha-1)-s+1$ and $1<\alpha<2$.}
Assume that there exist constants $0\le\delta<2^{2-\alpha}c$ such that
\begin{equation}\label{42.c}
  c \le c_x(n+1)-c_x(n) \le c+\delta \quad\mbox{for }x\in\{1,2,\ldots,G\},\ n\ge 0.
\end{equation}
Then the Beckner inequality \eqref{1.beck} and the decay estimates \eqref{1.exp}
and \eqref{3.decay} hold with $\lambda=\alpha c-(3+2^{\alpha-2}-\alpha)\delta$.
\end{theorem}

\begin{remark}\rm
In the case of constant rates, the spectral gap is of the order of $L^2/(L^2+N^2)$
\cite{Mor06}. Note that our bound $\lambda=2(c-\delta)$ for $\alpha=2$
does not depend on either $L$ or $N$. It was shown in 
\cite{BCDP06} that the lower bound in \eqref{42.c} is sufficient to derive
the spectral-gap estimate $\lambda\ge c$. In the homogeneous case $\delta=0$,
we have even $\lambda=2c$. As pointed out in \cite{CDP09},
a condition on the growth of the rates is necessary for the modified 
logarithmic Sobolev inequality. Our bound $\lambda=c-5\delta/2$ for $\alpha\to 1$
is the same as in \cite[Theorem 4.3]{FaMa15}.
\qed
\end{remark}

\begin{proof}
We define as in \cite[Section 4]{CDP09} the function
$$
  R(\eta,xy,uv) = \frac{1}{L^2}\left\{\begin{array}{ll}
	c_x(\eta_x)c_u(\eta_u) & \mbox{for }x\neq u, \\
	c_x(\eta_x)c_u(\eta_u-1) & \mbox{for }x = u,
	\end{array}\right.
$$
which satisfies Assumption \ref{ass}.
It follows that $\Gamma(\eta,xy,uv)=0$ if $x\neq u$ and
$$
  \Gamma(\eta,xy,uv) = L^{-2}c_x(\eta_x)\big(c_x(\eta_x)-c_x(\eta_x-1)\big)
	\quad\mbox{if }x=u,
$$ 
and the left-hand side of \eqref{ineq} can be written as
\begin{align*}
  \pi\bigg[ & \sum_{\gamma,\delta\in G}\Gamma(\eta,\gamma,\delta)\Big(
	\na_\gamma\rho^{\alpha-1}(\eta)\na_\delta\rho(\eta) + (\alpha-1)\na_\gamma\rho(\eta)
	\na_\delta\rho(\eta)\rho^{\alpha-2}(\eta)\Big)\bigg] \\
	&= \frac{1}{L^2}\pi\bigg[\sum_{x,y,v}c_x(\eta_x)\big(c_x(\eta_x)-c_x(\eta_x-1)\big)
	\na_{xv}\rho^{\alpha-1}(\eta)\na_{xy}\rho(\eta)\bigg] \\
	&\phantom{xx}{}
	+ \frac{\alpha-1}{L^2}\pi\bigg[\sum_{x,y,v}c_x(\eta_x)
	\big(c_x(\eta_x)-c_x(\eta_x-1)\big)
	\na_{xy}\rho(\eta)\na_{xv}\rho(\eta)\rho^{\alpha-2}(\eta)\bigg] \\
  &= C_1 + C_2.
\end{align*}
For future reference, we denote the right-hand side of \eqref{ineq} 
(without the constant $\lambda$) by
$$
  A = \frac{1}{2L}\pi\bigg[\sum_{x,y}c_x(\eta_x)\na_{xy}\rho^{\alpha-1}(\eta)
	\na_{xy}\rho(\eta)\bigg].
$$

The estimate of the term $C_1$ is similar to 
$\widetilde{\mathcal B}_1(\rho,\psi)$
in the proof of Theorem 4.3 in \cite{FaMa15} 
(take $\psi(\eta)=\rho^{\alpha-1}(\eta)$). 
First, we interchange $y$ and $v$ and then use
$\na_{xv}\rho^{\alpha-1}(\eta)=\na_{xy}\rho^{\alpha-1}(\eta)
+\na_{yv}\rho^{\alpha-1}(\eta^{xy})$ as well as the lower bound 
$c_x(\eta_x)-c_x(\eta_x-1)\ge c$:
\begin{align}
  C_1 &= \frac{1}{L^2}\pi\bigg[\sum_{x,y,v}c_x(\eta_x)
	\big(c_x(\eta_x)-c_x(\eta_x-1)\big)\big(\na_{xy}\rho^{\alpha-1}(\eta)
  +\na_{yv}\rho^{\alpha-1}(\eta^{xy})\big)\na_{xy}\rho(\eta)\bigg] \label{42.aux1} \\
  &\ge 2cA + \frac{1}{L^2}\pi\bigg[\sum_{x,y,v}c_x(\eta_x)
	\big(c_x(\eta_x)-c_x(\eta_x-1)\big)\na_{yv}\rho^{\alpha-1}(\eta^{xy})
	\na_{xy}\rho(\eta)\bigg]. \nonumber
\end{align}
Note that the term involving $\na_{xy}\rho^{\alpha-1}(\eta)$ does not
depend on $v$, so the sum over $x$, $y$, $v$ equals $L$ times the sum over
$x$, $y$. Employing the reversibility condition \eqref{42.db} and 
exchanging $x$ and $y$ in the second term yields
\begin{align}
  C_1 &\ge 2cA + \frac{1}{L^2}\pi\bigg[\sum_{x,y,v}c_y(\eta_y)
	\big(c_x(\eta_x^{yx})-c_x(\eta_x^{yx}-1)\big)\na_{yv}\rho^{\alpha-1}(\eta)
	\na_{xy}\rho(\eta^{yx})\bigg] \nonumber \\
	&= 2cA - \frac{1}{L^2}\pi\bigg[\sum_{x,y,v}c_x(\eta_x)
	\big(c_y(\eta_y+1)-c_y(\eta_y)\big)\na_{xv}\rho^{\alpha-1}(\eta)
	\na_{xy}\rho(\eta)\bigg]. \label{42.aux2}
\end{align}
We average \eqref{42.aux1} and \eqref{42.aux2} and employ again the identity
$\na_{xy}\rho^{\alpha-1}(\eta) + \na_{yv}\rho^{\alpha-1}(\eta^{xy})
= \na_{xv}\rho^{\alpha-1}(\eta)$:
\begin{align*}
  C_1 &\ge cA + \frac{1}{2L^2}\pi\bigg[\sum_{x,y,v}c_x(\eta_x)\Big(
	(c_x(\eta_x)-c_x(\eta_x-1)) - (c_y(\eta_y+1)-c_y(\eta_y))\Big) \\
	&\phantom{xx}{}\times\na_{xv}\rho^{\alpha-1}(\eta)\na_{xy}\rho(\eta)\bigg].	
\end{align*}
Setting $C_0:=(c_x(\eta_x)-c_x(\eta_x-1)) - (c_y(\eta_y+1)-c_y(\eta_y))$,
the bounds \eqref{42.c} imply that $|C_0| \le \delta$. 
Hence, by Young's inequality,
\begin{align*}
  C_0\na_{xv}\rho^{\alpha-1}(\eta)\na_{xy}\rho(\eta)
	&= C_0\widehat\rho(\eta,\eta^{xy})\na_{xv}\rho^{\alpha-1}(\eta)
	\na_{xy}\rho^{\alpha-1}(\eta) \\
	&\ge -\frac12|C_0|\widehat\rho(\eta,\eta^{xy})
	\Big((\na_{xy}\rho^{\alpha-1}(\eta))^2 + (\na_{xv}\rho^{\alpha-1}(\eta))^2\Big) \\
	&\ge -\frac{\delta}{2}\Big(\na_{xy}\rho(\eta)\na_{xy}\rho^{\alpha-1}(\eta)
  + (\na_{xv}\rho^{\alpha-1}(\eta))^2\widehat\rho(\eta,\eta^{xy})\Big).
\end{align*}
This yields
\begin{equation}\label{42.C1}
  C_1 \ge \left(c-\frac{\delta}{2}\right)A - \frac{\delta}{4L^2}\pi\bigg[
	\sum_{x,y,v}c_x(\eta_x)(\na_{xv}\rho^{\alpha-1}(\eta))^2\widehat\rho(\eta,\eta^{xy})
	\bigg].
\end{equation}

Next, we rewrite $B=(C_2-C_1)/2$. By definition \eqref{3.hatrho1} of
$\widehat\rho_1$ and the reversibility condition \eqref{42.db},
\begin{align*}
  B &= \frac{1}{2L^2}\pi\bigg[\sum_{x,y,v}c_x(\eta_x)
	\big(c_x(\eta_x)-c_x(\eta_x-1)\big)(\na_{xy}\rho^{\alpha-1}(\eta))^2
	\widehat\rho_1(\eta,\eta^{xy})\na_{xv}\rho(\eta)\bigg] \\
	&= \frac{1}{2L^2}\pi\bigg[\sum_{x,y,v}c_y(\eta_y)
	\big(c_x(\eta_x+1)-c_x(\eta_x)\big)(\na_{xy}\rho^{\alpha-1}(\eta^{yx}))^2
	\widehat\rho_1(\eta^{yx},\eta)\na_{xv}\rho(\eta^{yx})\bigg] \\
  &= \frac{1}{2L^2}\pi\bigg[\sum_{x,y,v}c_x(\eta_x)
	\big(c_y(\eta_y+1)-c_y(\eta_y)\big)(\na_{xy}\rho^{\alpha-1}(\eta))^2
	\widehat\rho_2(\eta,\eta^{xy})\big(\rho(\eta^{xv})-\rho(\eta^{xy})\big)\bigg].
\end{align*}
In the last step, we interchanged $x$ and $y$ and used the identity
$\widehat\rho_1(\eta^{xy},\eta) = \widehat\rho_2(\eta,\eta^{xy})$.
Averaging the expressions for $B$ involving $\widehat\rho_1$ and
$\widehat\rho_2$ gives
\begin{align*}
  B &= \frac{1}{4L^2}\pi\bigg[\sum_{x,y,v}c_x(\eta_x)(\na_{xy}\rho^{\alpha-1}(\eta))^2
	\rho(\eta^{xv}) \\
	&\phantom{xxxx}{}\times\Big(\big(c_x(\eta_x)-c_x(\eta_x-1)\big)
	\widehat\rho_1(\eta,\eta^{xy})
	+ \big(c_y(\eta_y+1)-c_y(\eta_y)\big)\widehat\rho_2(\eta,\eta^{xy})\Big)\bigg] \\
	&\phantom{xx}{}- \frac{1}{4L^2}\pi\bigg[\sum_{x,y,v}c_x(\eta_x)
	(\na_{xy}\rho^{\alpha-1}(\eta))^2 \\
	&\phantom{xxxx}{}\times \Big(\big(c_x(\eta_x)-c_x(\eta_x-1)\big)
	\widehat\rho_1(\eta,\eta^{xy})\rho(\eta) + \big(c_y(\eta_y+1)-c_y(\eta_y)\big)
	\widehat\rho_2(\eta,\eta^{xy})\rho(\eta^{xy})\Big)\bigg] \\
	&= B_1 + B_2.
\end{align*}
The term $B_1$ is estimated by using condition \eqref{42.c} (note that
$\widehat\rho_1$, $\widehat\rho_2\ge 0$ since $\theta$ is nondecreasing in
both variables) and then employing the assumption $c\ge 2^{\alpha-2}\delta$
and interchanging $y$ and $v$:
\begin{align*}
  B_1 &\ge \frac{c}{4L^2}\pi\bigg[\sum_{x,y,v}c_x(\eta_x)
	(\na_{xy}\rho^{\alpha-1}(\eta))^2\rho(\eta^{xv})\big(\widehat\rho_1(\eta,\eta^{xy})
	+ \widehat\rho_2(\eta,\eta^{xy})\big)\bigg] \\
	&\ge \frac{2^{\alpha-2}\delta}{4L^2}\pi\bigg[\sum_{x,y,v}c_x(\eta_x)
	(\na_{xv}\rho^{\alpha-1}(\eta))^2\rho(\eta^{xy})\Big(\widehat\rho_1(\eta,\eta^{xv})
	+ \widehat\rho_2(\eta,\eta^{xv})\Big)\bigg] \\
  &= B_3.
\end{align*}
We employ condition \eqref{42.c} once more and Lemma \ref{lem.theta2} (i) (see
Appendix) to estimate $B_2$:
\begin{align*}
  B_2 &\ge -\frac{c+\delta}{4L^2}\pi\bigg[\sum_{x,y,v}c_x(\eta_x)
	(\na_{xy}\rho^{\alpha-1}(\eta))^2\Big(\widehat\rho_1(\eta,\eta^{xy})\rho(\eta)
	+ \widehat\rho_2(\eta,\eta^{xy})\rho(\eta^{xy})\Big)\bigg] \\
	&= -\frac{c+\delta}{4L^2}(2-\alpha)\pi\bigg[\sum_{x,y,v}c_x(\eta_x)
	(\na_{xy}\rho^{\alpha-1}(\eta))^2\widehat\rho(\eta,\eta^{xy})\bigg] 
	= -\frac12(2-\alpha)(c+\delta)A.
\end{align*}
Consequently, 
\begin{equation}\label{42.B}
  B \ge -\frac12(2-\alpha)(c+\delta)A + B_3.
\end{equation}
We add \eqref{42.C1} and \eqref{42.B}:
\begin{align}
  & C_1 + B \ge \left(c-\frac{\delta}{2}-\frac12(2-\alpha)(c+\delta)\right)A + B_4, 
	\quad\mbox{where} \label{42.C1B} \\
	& B_4 = B_3 - \frac{\delta}{4L^2}\pi\bigg[
	\sum_{x,y,v}c_x(\eta_x)(\na_{xv}\rho^{\alpha-1}(\eta))^2\widehat\rho(\eta,\eta^{xy})
	\bigg]. \nonumber
\end{align}

We wish to estimate $B_4$ from below by a multiple of $A$. To this end, we employ
the reversibility and interchange $x$ and $v$ in the second term in $B_4$:
\begin{align*}
  \pi\bigg[ \sum_{x,y,v}c_x(\eta_x)(\na_{xv}\rho^{\alpha-1}(\eta))^2
	\widehat\rho(\eta,\eta^{xy})\bigg] 
	&= \pi\bigg[\sum_{x,y,v}c_v(\eta_v)(\na_{xv}\rho^{\alpha-1}(\eta^{vx}))^2
	\widehat\rho(\eta^{vx},\eta^{vy})\bigg] \\
  &= \pi\bigg[\sum_{x,y,v}c_x(\eta_x)(\na_{xv}\rho^{\alpha-1}(\eta))^2
	\widehat\rho(\eta^{xv},\eta^{xy})\bigg].
\end{align*}
Then, averaging those two expressions for $B_4$ that involve 
$\widehat\rho(\eta,\eta^{xy})$ and $\widehat\rho(\eta^{xv},\eta^{xy})$,
\begin{align*}
  B_4 &= \frac{\delta}{8L^2}\pi\bigg[\sum_{x,y,v}c_x(\eta_x)
	(\na_{xv}\rho^{\alpha-1}(\eta))^2(2^{\alpha-1}\rho(\eta^{xy}))
	\Big(\widehat\rho_1(\eta,\eta^{xv}) + \widehat\rho_2(\eta,\eta^{xv})\Big)\bigg] \\
	&\phantom{xx}{}- \frac{\delta}{8L^2}\pi\bigg[\sum_{x,y,v}c_x(\eta_x)
	(\na_{xv}\rho^{\alpha-1}(\eta))^2\Big(\widehat\rho(\eta,\eta^{xy})
	+ \widehat\rho(\eta^{xv},\eta^{xy})\Big)\bigg].
\end{align*}
We employ Lemma \ref{lem.theta2} (ii) in the form
\begin{align*}
  2^{\alpha-1}\rho(\eta^{xy}))\big(\widehat\rho_1(\eta,\eta^{xv}) 
	+ \widehat\rho_2(\eta,\eta^{xv})\big)
	- \big(\widehat\rho(\eta,\eta^{xy})	+ \widehat\rho(\eta^{xv},\eta^{xy})\big)
  \ge -2^{\alpha-1}\widehat\rho(\eta,\eta^{xv}),
\end{align*}
which leads to
$$
  B_4 \ge -\frac{2^{\alpha-1}\delta}{8L^2}\pi\bigg[\sum_{x,y,v}c_x(\eta_x)
	(\na_{xv}\rho^{\alpha-1}(\eta))^2\widehat\rho(\eta,\eta^{xv})\bigg]
	= -\frac{2^{\alpha-1}\delta}{4}A.
$$
Hence, we infer from \eqref{42.C1B} that
$$
  C_1 + B \ge \left(c - \frac{\delta}{2} - \frac12(2-\alpha)(c+\delta) 
	- \frac{\delta}{4}2^{\alpha-1}\right)A.
$$
Finally, by definition of $B$,
$$
  C_1 + C_2 = 2(C_1+B) 
	\ge \big(2c - \delta - (2-\alpha)(c+\delta) - 2^{\alpha-2}\delta\big)A = \lambda A.
$$
This shows \eqref{ineq}, and an application of Corollary \ref{coro} finishes the 
proof.
\end{proof}


\subsection{Bernoulli-Laplace models}

We consider again a system of particles moving in a finite set of sizes
$\{1,2,\ldots,L\}$ but in contrast to the previous subsection, we assume
that at most one particle per site is allowed, i.e.\ $S=\{0,1\}^L$. 
The set of allowed moves is $G=\{xy:x,y\in\{1,2,\ldots,L\}$, $x\neq y\}$,
and the moves are of the form $xy:\eta\mapsto \eta^{xy}$ for $\eta\in S$, where
$\eta^{xy}=\eta$ if $\eta_x(1-\eta_y)=0$ and otherwise,
$$
  \eta^{xy}_z = \left\{\begin{array}{ll}
	\eta_z &\mbox{if }z\not\in\{x,y\}, \\
	0 &\mbox{for }z=x, \\
	1 &\mbox{for }z=y.
	\end{array}\right.
$$
We associate to each site $x$ a Poisson clock of constant intensity $\lambda_x>0$.
When the clock of site $x$ rings, we choose randomly a site $y$. If $\eta_x=1$ and
$\eta_y=0$ (i.e.\ if $\eta_x(1-\eta_y)=1$), the particle at $x$ moves to $y$; 
otherwise (i.e.\ if $\eta_x(1-\eta_y)=0$), nothing happens.
Therefore, the transition rates are given by 
$c(\eta,xy) = (\lambda_x/L)\eta_x(1-\eta_y)$, and the generator reads as
$$
  \L f(\eta) = \frac{1}{L}\sum_{xy\in G}\lambda_x\eta_x(1-\eta_y)\na_{xy} f(\eta),
$$
where, as in the previous subsection, $\na_{xy}f(\eta)=f(\eta^{xy})-f(\eta)$.

Let $N\le L$ be the number of particles in the system. There exists a unique
stationary distribution $\pi_N$, which is given by \cite[Section 5]{CDP09}
$$
  \pi_N(\eta) = \frac{1}{Z_{L,N}}\prod_{x=1}^L
	\left(\frac{1}{1+\lambda_x}\right)^{\eta_x}
	\left(\frac{\lambda_x}{1+\lambda_x}\right)^{1-\eta_x},
$$
where $Z_{L,N}>0$ is a normalization constant. In the following,
we write $\pi$ instead of $\pi_N$, as the number of particles is fixed.
Reversibility holds for $\pi$, and it reads as
\begin{equation}\label{43.db}
  \pi\bigg[\sum_{xy\in G}c(\eta,xy)F(\eta,xy)\bigg]
	= \pi\bigg[\sum_{xy\in G}c(\eta,xy)F(\eta^{xy},yx)\bigg]
\end{equation}
for arbitrary functions $F:S\times G\to\R$. 

\begin{theorem}\label{thm.blm}
\blue{Let $\phi(s)=(s^\alpha-s)/(\alpha-1)-s+1$ and $1<\alpha<2$.}
Assume that there exist constants $0\le\delta\le 2^{2-\alpha}c$ such that
\begin{equation}\label{43.c}
  c \le \lambda_x\le c+\delta \quad\mbox{for }x\in\{1,2,\ldots,L\}.
\end{equation}
Then the Beckner inequality \eqref{1.beck} and the decay estimates \eqref{1.exp}
and \eqref{3.decay} hold with $\lambda=\alpha c-(\frac52+2^{\alpha-3}-\alpha)\delta$.
\end{theorem}

\begin{remark}\rm
For the modified log-Sobolev inequality, the bound in \cite{CDP09} 
reads as $\lambda=c-\delta$, and the bound in \cite{FaMa15} equals
$\lambda=c-7\delta/4$ (for $\delta<4c/7$). Our result coincides with that in
\cite{FaMa15} for $\alpha\to 1$. In \cite{GaQu03}, the bound
$1\le\lambda\le 2$ was proved in case $c=1$, $\delta=0$. 
Further bounds, depending on $L$ and $N$,
were collected in \cite[Examples 3.11]{BoTe06}. 

Concerning the Beckner inequality,
Bobkov and Tetali \cite[Section 4]{BoTe06} derived for the homogeneous case
$c=L/(N(L-N))$ and $\delta=0$ the constant $\lambda\ge \alpha(L+2)/(2N(L-N)$.
Our constant $\lambda=(\alpha L-2\alpha+4)/(N(L-N))$ (see the proof below)
is larger for $L>2$ and all $1<\alpha\le 2$.
\qed
\end{remark}

\begin{proof}
We need to verify the condition in Corollary \ref{coro}. As in \cite{CDP09},
we choose
$$
  R(\eta,xy,uv) = L^{-2}\lambda_x\lambda_u\eta_x(1-\eta_y)\eta_u(1-\eta_v)
	\quad\mbox{for }|\{x,y,u,v\}|=4
$$
and $R(\eta,xy,uv)=0$ otherwise. The notation $|\{x,y,u,v\}|=4$ means that
the four variables are pairwise different. Then
$\Gamma(\eta,xy,uv)=0$ if $|\{x,y,u,v\}|=4$ and 
$$
  \Gamma(\eta,xy,uv) =  L^{-2}\lambda_x\lambda_u\eta_x(1-\eta_y)\eta_u(1-\eta_v)
$$
otherwise.
The sum of $\Gamma(\eta,\gamma,\delta)$ over $\gamma$, $\delta\in G$
in the left-hand side of \eqref{ineq} vanishes if $(x,y,u,v)$ are pairwise
different. Therefore, the sum consists of three terms: $(\gamma,\delta)=(xy,xy)$,
$(\gamma,\delta)=(xy,uy)$, and $(\gamma,\delta)=(xy,xv)$, and it follows that
\begin{align*}
  \pi\bigg[ & \sum_{\gamma,\delta\in G}\Gamma(\eta,\gamma,\delta)
	\Big(\na_\gamma\rho^{\alpha-1}(\eta)\na_\delta\rho(\eta)
	+ (\alpha-1)\na_\gamma\rho(\eta)\na_\delta\rho(\eta)\rho^{\alpha-2}(\eta)
	\Big)\bigg] \\
	&= \frac{1}{L^2}\pi\bigg[\sum_{x,y}\lambda_x^2\na_{xy}\rho^{\alpha-1}(\eta)
	\na_{xy}\rho(\eta) 
	+ \sum_{|\{x,y,u\}|=3}\lambda_x\lambda_u\na_{xy}\rho^{\alpha-1}(\eta)
	\na_{uy}\rho(\eta) \\
	&\phantom{xx}{}+ \sum_{|\{x,y,v\}|=3}\lambda_x^2\na_{xy}\rho^{\alpha-1}(\eta)
	\na_{xv}\rho(\eta)\bigg] 
	+ \frac{\alpha-1}{L^2}\pi\bigg[\sum_{x,y}\lambda_x^2\na_{xy}\rho(\eta)
	\na_{xy}\rho(\eta)\rho^{\alpha-2}(\eta) \\
	&\phantom{xx}{}
	+ \sum_{|\{x,y,u\}|=3}\lambda_x\lambda_u\na_{xy}\rho(\eta)\na_{uy}\rho(\eta)
	\rho^{\alpha-2}(\eta)
	+ \sum_{|\{x,y,v\}|=3}\lambda_x^2\na_{xy}\rho(\eta)
	\na_{xv}\rho(\eta)\rho^{\alpha-2}(\eta)\bigg] \\
	&= C_1 + C_2.
\end{align*}

Observe that the right-hand side of \eqref{ineq} (without the constant $\lambda$)
reads as
\begin{equation}\label{43.A}
  A = \frac12\pi\bigg[\sum_{\gamma\in G}c(\eta,\gamma)\na_\gamma\rho^{\alpha-1}(\eta)
	\na_\gamma\rho(\eta)\bigg]
	= \frac{1}{2L}\pi\bigg[\sum_{xy\in G}\lambda_x\na_{xy}\rho^{\alpha-1}(\eta)
	\na_{xy}\rho(\eta)\bigg],
\end{equation}
since $\na_{xy}\rho(\eta)=0$ whenever $\eta_x(1-\eta_y)=0$, so the factor
$\eta_x(1-\eta_y)$ can be omitted.

As in the previous subsection, we estimate $B=(C_2-C_1)/2$, recalling definition
\eqref{3.hatrho1} of $\widehat\rho_1$:
\begin{align}
  B &= \frac{1}{2L^2}\pi\bigg[\sum_{x,y}\lambda_x^2(\na_{xy}\rho^{\alpha-1}(\eta))^2
	\widehat\rho_1(\eta,\eta^{xy})\na_{xy}\rho(\eta)\bigg] \label{43.B1} \\
	&\phantom{xx}{}+ \frac{1}{2L^2}\pi\bigg[ \sum_{|\{x,y,u\}|=3}\lambda_x\lambda_u
	(\na_{xy}\rho^{\alpha-1}(\eta))^2\widehat\rho_1(\eta,\eta^{xy})
	\na_{uy}\rho(\eta)\bigg] \nonumber \\
	&\phantom{xx}{}+ \frac{1}{2L^2}\pi\bigg[ \sum_{|\{x,y,v\}|=3}\lambda_x^2
	(\na_{xy}\rho^{\alpha-1}(\eta))^2\widehat\rho_1(\eta,\eta^{xy})
	\na_{xv}\rho(\eta)\bigg] \nonumber \\
  &= B_1 + B_2 + B_3. \nonumber
\end{align}

The estimations of $B_1$, $B_2$, and $B_3$
are the same as in the proof of Theorem 4.6 in \cite{FaMa15}
after taking $\psi(\eta)=\rho^{\alpha-1}(\eta)$ in 
$\widetilde{\mathcal B}_2(\rho,\psi)$. 
The key point is the use of Lemma \ref{lem.theta2} (iii). In constrast to
\cite{FaMa15}, the factor $2-\alpha$ appears. Therefore, following \cite{FaMa15}
and taking into account \eqref{43.A}, we conclude that
\begin{align}
  B_1 &\ge -\frac{\delta}{2L}(2-\alpha)A, \nonumber \\
	B_2 &\ge -\frac{1}{2L}(N-1)(c+\delta)(2-\alpha)A, \nonumber \\
	B_3 &\ge \frac{c}{4L^2}\pi\bigg[\sum_{|\{x,y,v\}|=3}\lambda_x\eta_x(1-\eta_y)
	(1-\eta_v)(\na_{xy}\rho^{\alpha-1}(\eta))^2\rho(\eta^{xv})
	\nonumber \\
	&\phantom{xx}{}\times\Big(\widehat\rho_1(\eta,\eta^{xy}) 
	+ \widehat\rho_2(\eta,\eta^{xy})\Big)\bigg] 
	- \frac{1}{2L}(L-N-1)(c+\delta)(2-\alpha)A. \label{43.B3}
\end{align}
Since we assumed that $\delta\le 2^{2-\alpha}c$, we can estimate the factor in the
first term of $B_3$ by $c/(4L^2)\ge 2^{\alpha-4}\delta/L^2$.

Next, we estimate $C_1$. This expression consists of three terms. We interchange
$x$ and $u$ in the second term and $y$ and $v$ in the third term. Then
$C_1=B_4+B_5+B_6$, where
\begin{align*}
  B_4 &= \frac{1}{L^2}\pi\bigg[\sum_{x,y}\lambda_x^2\na_{xy}\rho(\eta)
	\na_{xy}\rho^{\alpha-1}(\eta)\bigg], \\
	B_5 &= \frac{1}{L^2}\pi\bigg[\sum_{|\{x,y,u\}|=3}\lambda_x\lambda_u
	\na_{xy}\rho(\eta)\na_{uy}\rho^{\alpha-1}(\eta)\bigg], \\
	B_6 &= \frac{1}{L^2}\pi\bigg[\sum_{|\{x,y,v\}|=3}\lambda_x^2\na_{xy}\rho(\eta)
	\na_{xv}\rho^{\alpha-1}(\eta)\bigg].
\end{align*}
By condition \eqref{43.c}, $B_4\ge (2c/L)A$. 
The term $B_6$ is estimated by employing the reversibility \eqref{43.db},
averaging, and using \eqref{43.c}, similar to the estimate of $J_6$
in the proof of Theorem 4.6 in \cite{FaMa15}. The result is
\begin{align}
  & B_6 \ge \frac{1}{2L}(L-N-1)(2c-\delta)A - B_7, \quad\mbox{where} \label{43.B6} \\
	& B_7 = \frac{\delta}{4L^2}\pi\bigg[\sum_{|\{x,y,v\}|=3}\lambda_x
	\eta_x(1-\eta_y)(1-\eta_v)(\na_{xv}\rho^{\alpha-1}(\eta))^2
	\widehat\rho(\eta,\eta^{xy})\bigg]. \nonumber
\end{align}
Similarly, replacing $\psi(\eta)$ by $\rho^{\alpha-1}(\eta)$ in $J_5$ in the
proof of Theorem 4.6 in \cite{FaMa15}, we have
$B_5\ge (c/L)(N-1)A$.

It remains to rewrite $B_7$. For this, we employ the reversibility, average
the original and the resulting expressions, and interchange $y$ and $v$. This
yields (see the computation of $J_7$ in \cite{FaMa15})
$$
  B_7 = \frac{\delta}{8L^2}\pi\bigg[\sum_{|\{x,y,v\}|=3}\lambda_x\eta_x(1-\eta_y)
	(1-\eta_v)(\na_{xy}\rho^{\alpha-1}(\eta))^2\Big(\widehat\rho(\eta^{xv},\eta^{xy})
	+ \widehat\rho(\eta,\eta^{xv})\Big)\bigg].
$$
Combining estimate \eqref{43.B3} for $B_3$ and \eqref{43.B6}, together with
the above estimate for $B_7$ and applying Lemma \ref{lem.theta2} (ii), we infer that
\begin{align*}
  B_3+B_6 &\ge \frac{1}{2L}(L-N-1)\big(\alpha c-(3-\alpha)\delta\big)A \\
	&\phantom{xx}{}
	+ \frac{\delta}{8L^2}\pi\bigg[\sum_{|\{x,y,v\}|=3}\lambda_x\eta_x(1-\eta_y)
	(1-\eta_v)(\na_{xy}\rho^{\alpha-1}(\eta))^2 \\
	&\phantom{xxxx}{}\times\Big(2^{\alpha-1}
	\rho(\eta^{xv})\big(\widehat\rho_1(\eta,\eta^{xy})
  + \widehat\rho_2(\eta,\eta^{xy})\big) - \big(\widehat\rho(\eta^{xv},\eta^{xy})
	+ \widehat\rho(\eta^{xv},\eta)\big)\Big)\bigg] \\
	&\ge \frac{1}{4L}(L-N-1)
	\big(2\alpha c - 2(3-\alpha)\delta - 2^{\alpha-1}\delta\big)A.
\end{align*}

It remains to summarize the estimates:
\begin{align*}
  C_1 + C_2 &= 2B+2C_1 = 2(B_1+B_2) + 2(B_4+B_5) + 2(B_3+B_6) \\
	&\ge -\frac{(2-\alpha)}{L}\big(\delta+(N-1)(c+\delta)\big)A 
	+ \frac{2}{L}\big(2c+(N-1)c\big)A \\
	&\phantom{xx}{}+ \frac{1}{2L}(L-N-1)\big(2\alpha c - 2(3-\alpha)\delta 
	- 2^{\alpha-1}\delta\big)A \\
	&= \frac{1}{L}\Big((\alpha L+4-2\alpha)c + \big((\alpha-2^{\alpha-2}-3)L 
	+ (1+2^{\alpha-2})N
	+ (3+2^{\alpha-2}-\alpha)\big)\delta\Big)A.
\end{align*}
Arguing as in \cite{FaMa15}, we may suppose that $N\ge L/2$. 
Because of $4-2\alpha\ge 0$, $(1+2^{\alpha-2})N/L \ge (1+2^{\alpha-2})/2$,
and $3+2^{\alpha-2}-\alpha\ge 0$, we infer that
\begin{align*}
  C_1+C_2 
	&\ge \left(\frac{1}{L}(\alpha L+4-2\alpha)c 
	+ \left(\alpha-\frac52-2^{\alpha-3}\right)\delta\right)A \\
  &\ge \left(\alpha c - \left(\frac52+2^{\alpha-3}-\alpha\right)\delta\right)A
\end{align*}
which concludes the proof.
\end{proof}


\subsection{Random transposition model}\label{sec.rtm}

The random transposition model is a random walk on the group of permutations.
Let $S_n$ be the set of permutations on $\{1,2,\ldots,n\}$ and
$T_n$ the set of all transpositions in $S_n$. Given $1\le i,j\le n$,
we denote by $\tau_{ij}\in T_n$ the transposition that interchanges $i$ and $j$,
i.e.\ $\tau_{ij}(i)=j$, $\tau_{ij}(j)=i$, and $\tau_{ij}(k)=k$ for $k\neq i,j$.
The composition of two permutations $\sigma_1$, $\sigma_2\in S_n$ is denoted
by $\sigma_1\sigma_2$.

We define a graph structure on the group $S_n$ by saying that two permutations
are neighbors if they differ by precisely one transposition. Thus every vertex
$\sigma\in S_n$ has $\binom{n}{2}=n(n-1)/2$ neighbors given by 
$\{\tau_{ij}\sigma\}_{1\le i,j\le n}$, and the set of edges is 
$E_n=\{\{\sigma,\tau_{ij}\sigma\}:1\le i,j\le n$, $\sigma\in S_n\}$.
We write $\sigma\leftrightarrow\tau\sigma$ if $\{\sigma,\tau\sigma\}\in E_n$.
The random walk on $(S_n,E_n)$ is then defined by the transition rates
$c(\sigma,\tau)=2/(n(n-1))$ if $\sigma\leftrightarrow\tau\sigma$ and
$c(\sigma,\tau)=0$ otherwise. 
The generator of the Markov chain reads as
$$
  \L f(\sigma) = \frac{2}{n(n-1)}\sum_{\tau\in T_n}\na_\tau f(\sigma),
$$
where $\na_\tau f(\sigma) = f(\tau\circ\sigma)-f(\sigma)$. The uniform
measure $\pi(\sigma)=1/n!$ for all $\sigma\in S_n$ is reversible for the above
transition rates $c(\sigma,\tau)$. 
To simplify the notation, we write $\na_{ij}=\na_\tau$ if $\tau=\tau_{ij}$,
$\sigma_{ij}=\tau_{ij}\circ\sigma$, and $\sigma_{ijk}=\tau_{ij}\circ\tau_{jk}
\circ\sigma$.

\begin{theorem}\label{thm.rtm}
\blue{Let $\phi(s)=(s^\alpha-s)/(\alpha-1)-s+1$ and $1<\alpha<2$.}
For $n\ge 2$, the Beckner inequality \eqref{1.beck} and the decay estimates 
\eqref{1.exp} and \eqref{3.decay} hold with constant $\lambda=8/(n(n-1))$.
\end{theorem}

\begin{remark}\rm
Diaconis and Saloff-Coste \cite[Section 4.3]{DiSa96} report that
the logarithmic So\-bo\-lev constant satisfies the bounds 
$1/(3n\log n)\le\lambda\le 1/(n-1)$; also see \cite[Theorem 1]{GaQu03}. 
Our bound is worse by a factor of $1/n$. The bound $\lambda\ge \alpha(n+2)/(n(n-1))$
was derived in \cite[Section 4]{BoTe06}. It is usually better than our bound
$\lambda=8/(n(n-1))$; for very small numbers of $n$ (namely $n<(8/\alpha)-2$),
our result is superior.
\qed
\end{remark}

\begin{proof}
The right-hand side of \eqref{ineq} (except the factor $\lambda$) can be written as
\begin{equation}\label{44.A}
  A = \frac{1}{n(n-1)}\pi\bigg[\sum_{\tau\in T_n}\na_\tau\rho^{\alpha-1}(\sigma)
	\na_\tau\rho(\sigma)\bigg]
	= \frac{1}{2n(n-1)}\pi\bigg[\sum_{i\neq j}\na_{ij}\rho^{\alpha-1}(\sigma)
	\na_{ij}\rho(\sigma)\bigg],
\end{equation}
where the factor $1/2$ takes into account that every transposition $(i,j)$
is counted twice. As in \cite[Section 4.4]{FaMa15}, we define
$R(\sigma,(i,j),(k,\ell))=4/(n^2(n-1)^2)$ if $|\{i,j,k,\ell\}|=4$ and
$R(\sigma,(i,j),(k,\ell))=0$ otherwise. Then
$\Gamma(\sigma,(i,j),(k,\ell))=0$ if $|\{i,j,k,\ell\}|=4$ and
$$
  \Gamma(\sigma,(i,j),(k,\ell)) = \frac{4}{n^2(n-1)^2}
$$
otherwise. The left-hand side of \eqref{ineq} then becomes
\begin{align*}
  \pi\bigg[ & \sum_{\gamma,\delta}\Gamma(\sigma,\gamma,\delta)\Big(
	\na_\gamma\rho^{\alpha-1}(\sigma)\na_\delta\rho(\sigma) + (\alpha-1)
	\na_\gamma\rho(\sigma)\na_\delta\rho(\sigma)\rho^{\alpha-2}(\sigma)\bigg] \\
	&= \frac{2}{n^2(n-1)^2}\pi\bigg[\sum_{i\neq j}\na_{ij}\rho^{\alpha-1}(\sigma)
	\na_{ij}\rho(\sigma) + 2\sum_{|\{i,j,k\}|=3}\na_{ij}\rho(\sigma)
	\na_{ik}\rho^{\alpha-1}(\sigma)\bigg] \\
	&\phantom{xx}{} + \frac{2(\alpha-1)}{n^2(n-1)^2}\pi\bigg[
	\sum_{i\neq j}\na_{ij}\rho(\sigma)\na_{ij}\rho(\sigma)\rho^{\alpha-2}(\sigma)
	+ 2\sum_{|\{i,j,k\}|=3}\na_{ij}\rho(\sigma)\na_{ik}\rho(\sigma)
	\rho^{\alpha-2}(\sigma)\bigg] \\
	&= C_1 + C_2.
\end{align*}
The expression $C_1$ can be estimated exactly as in the proof of Theorem 4.8
in \cite{FaMa15} using the reversibility and averaging
(see the estimate for $\widetilde{\mathcal B}_1(\rho,\psi)$
for $\psi=\rho^{\alpha-1}$):
$$
  C_1 \ge \frac{2}{n-1}A - \frac{1}{n^2(n-1)^2}\pi\bigg[\sum_{|\{i,j,k\}|=3}
	\big(\rho^{\alpha-1}(\sigma_{ij})-\rho^{\alpha-1}(\sigma)\big)^2
	\widehat\rho(\sigma_{ik},\sigma_{ijk})\bigg].
$$

We estimate now $B=(C_2-C_1)/2$:
\begin{align*}
  B &= \frac{1}{n^2(n-1)^2}\pi\bigg[\sum_{i\neq j}(\na_{ij}\rho^{\alpha-1}(\sigma))^2
	\na_{ij}\rho(\sigma)\widehat\rho_1(\sigma,\sigma_{ij}) \\
	&\phantom{xx}{}+ 2\sum_{|\{i,j,k\}|=3}(\na_{ik}\rho^{\alpha-1}(\sigma))^2
	\na_{ij}\rho(\sigma)\widehat\rho_1(\sigma,\sigma_{ik})\bigg].
\end{align*}
Arguing as for $\widetilde{\mathcal B}_2(\rho,\psi)$ with $\psi=\rho^{\alpha-1}$
in the proof of Theorem 4.8 in \cite{FaMa15}, it follows that
\begin{align*}
  B &= \frac{1}{n^2(n-1)^2}\pi\bigg[\sum_{|\{i,j,k\}|=3}
	(\na_{ij}\rho^{\alpha-1}(\sigma))^2
	\Big(\rho(\sigma_{ik})\widehat\rho_1(\sigma,\sigma_{ij})
	+ \rho(\sigma_{ijk})\widehat\rho_2(\sigma,\sigma_{ij})\Big)\bigg] \\
	&\phantom{xx}{}- \frac{1}{n^2(n-1)^2}\pi\bigg[\sum_{|\{i,j,k\}|=3}
	(\na_{ij}\rho^{\alpha-1}(\sigma))^2
	\Big(\rho(\sigma)\widehat\rho_1(\sigma,\sigma_{ij})
	+ \rho(\sigma_{ij})\widehat\rho_2(\sigma,\sigma_{ij})\Big)\bigg] \\
	&\phantom{xx}{}+ \frac{1}{2n^2(n-1)^2}\pi\bigg[\sum_{i\neq j}
	(\na_{ij}\rho^{\alpha-1}(\sigma))^2
	\na_{ij}\rho(\sigma)\Big(\widehat\rho_1(\sigma,\sigma_{ij})
	- \widehat\rho_2(\sigma,\sigma_{ij})\Big)\bigg] \\
	&= B_1 + B_2 + B_3.
\end{align*}
Property (iii) of Lemma \ref{lem.theta2} (applied with $\lambda_1=\lambda_2=1$) 
implies that $B_3\ge 0$.
Combining $B_1$ and $B_2$, we can apply Lemma \ref{lem.phi1} with 
$s=\rho(\sigma)$, $t=\rho(\sigma_{ij})$, $u=\rho(\sigma_{ik})$, and
$v=\rho(\sigma_{ijk})$, leading to
\begin{align*}
  B &\ge B_1 + B_2 \ge \frac{1}{n^2(n-1)^2}\pi\bigg[\sum_{|\{i,j,k\}|=3}
	(\na_{ij}\rho^{\alpha-1}(\sigma))^2\big(\widehat\rho(\sigma_{ik},\sigma_{ijk})
	-\widehat\rho(\sigma,\sigma_{ij})\big)\bigg] \\
	&= \frac{1}{n^2(n-1)^2}\pi\bigg[\sum_{|\{i,j,k\}|=3}
	(\na_{ij}\rho^{\alpha-1}(\sigma))^2\widehat\rho(\sigma_{ik},\sigma_{ijk})\bigg]
	- \frac{2(n-2)}{n(n-1)}A.
\end{align*}
Adding the estimations for $C_1$ and $B$, one term cancels and we end up with
$$
  C_1+C_2 = 2(C_1+B)
	\ge 2\left(\frac{2}{n-1} - \frac{2(n-2)}{n(n-1)}
	\right)A = \frac{8}{n(n-1)}A.
$$
This concludes the proof.
\end{proof}


\section{Application: Finite-volume discretization of a Fokker-Planck 
equation}\label{sec.fp}

The Bakry-Emery method has been originally applied to Markov diffusion operators
or associated Fokker-Planck equations, and the exponential decay for the
probability densities with an explicit decay rate was shown. In numerical
analysis, the aim is to prove this equilibration property for numerical
discretizations of Fokker-Planck equations. As these discretizations can,
at least in some cases, be interpreted as a Markov chain, one may apply
Markov chain theory to achieve this goal. This was done by Mielke 
\cite[Section 5.3]{Mie13} to prove exponential decay of the logarithmic entropy
for a finite-volume approximation of a Fokker-Planck equation.
The proof is based on diagonal dominance properties of the matrices
appearing in \eqref{1.ineq}.
Our aim is to extend the exponential decay to power-type entropies by combining 
Mielkes results and the estimate for birth-death processes from Theorem \ref{thm.bdp}.
As a by-product, this provides an alternative proof for the
case $\alpha\to 1$ without using matrix algebra.

More specifically,
we consider a finite-volume approximation of the one-dimensional 
Fokker-Planck equation
\begin{equation}\label{5.fpe}
  \pa_t u = \pa_x(\pa_x u + u\pa_x V), \quad t>0, 
	\quad u(\cdot,0)=u_0 \quad\mbox{in }\R,
\end{equation}
where $u(x,t)$ describes some probability density and $V(x)$ is a given potential
satisfying $e^{-V}\in L^1(\R)$.
We introduce the uniform grid $x_n=n/N$, $n\in\Z$, where $N\in\N$.
The quantity $h=1/N$ is the grid size. 
The Fokker-Planck equation has the unique steady state $\pi(x) = Ze^{-V(x)}$,
where $Z>0$ is a normalization constant. The symmetric form of \eqref{5.fpe}, 
$$
  \pa_t u = \pa_x\bigg(\pi\pa_x\bigg(\frac{u}{\pi}\bigg)\bigg),
$$
motivates the following numerical scheme. We integrate this equation over
$[x_{n-1},x_n]$:
$$
  \frac{d}{dt}\,\frac{1}{h}\int_{x_{n-1}}^{x_n}u(x,t)dx
	= \frac{1}{h}\bigg[\pi\pa_x\bigg(\frac{u(\cdot,t)}{\pi}\bigg)\bigg]_{x_{n-1}}^{x_n}.
$$
We choose $u_n$ to approximate $\int_{x_{n-1}}^{x_n}u(\cdot,x)dx/h$,
$\pi_n=\int_{x_{n-1}}^{x_n}\pi(x)dx/h$, and the numerical flux $q_n$ to approximate
$h^{-1}[\pi\pa_x(u/\pi)](x_n)$. We choose as in \cite{Mie13}
$$
  q_n = \frac{\kappa_n}{h^2}\bigg(\frac{u_{n+1}}{\pi_{n+1}}-\frac{u_n}{\pi_n}\bigg), 
	\quad\kappa_n = (\pi_n\pi_{n+1})^{1/2}.
$$
Setting $\rho_n=u_n/\pi_n$, the numerical scheme reads as
\begin{align*}
  \pa_t \rho_n 
	&= \frac{1}{\pi_n}(q_n - q_{n-1}) 
	= \frac{\kappa_n}{h^2\pi_n}(\rho_{n+1}-\rho_n)
	+ \frac{\kappa_{n-1}}{h^2\pi_n}(\rho_{n-1}-\rho_n) \\
	&= a(n)\na_+\rho_n + b(n)\na_-\rho_n,
\end{align*}
where we employed the notation of Section \ref{sec.bdp} and
$a(n)=\kappa_n/(h^2\pi_n)$, $b(n)=\kappa_{n-1}/(h^2\pi_n)$.
The right-hand side can be interpreted as the generator of a birth-death
process on $\Z$. The initial datum is given by $\rho_n(0)=u_n(0)/\pi_n$,
where $u_n(0)=\int_{x_{n-1}}^{x_n}u(x,0)dx/h$.
According to \cite[Section 3.5]{CDP09}, the results of
Theorem \ref{thm.bdp} still hold in that case, and the assumption $b(0)=0$
is clearly not needed. The entropy is given by
$$
  \mathrm{Ent}_\pi^{\phi_\alpha}(\rho) 
	= \frac{1}{\alpha-1}\sum_{n\in\Z}\pi_n(\rho_n^\alpha-1), \quad
	\rho=(\rho_n)_{n\in\Z},\ 1<\alpha\le 2.
$$

\blue{
\begin{theorem}
Let $V\in C^2([0,1])$ and $V''(x)\ge\lambda>0$ for $x\in[0,1]$. 
Then 
$$
  \mathrm{Ent}_\pi^{\phi_\alpha}(\rho(t)) 
	\le \mathrm{Ent}_\pi^{\phi_\alpha}(\rho(0))e^{-2\alpha\lambda_h t}, \quad n\in\N,
$$
where $\lambda_h = 2h^{-2}\Phi(h^2\lambda/8)$ and
$$
  \Phi(s^2) = \frac{3\mathrm{erf}(s)-\mathrm{erf}(3s)}{2\mathrm{erf}(s)} \quad
	\mbox{with}\quad\mathrm{erf}(s)=\frac{2}{\sqrt{p}}\int_0^s e^{-t^2}dt
$$
and $p=3.14159\ldots$ is the number pi (to avoid confusion with the invariant measure
$\pi$). Moreover, the following discrete Beckner inequality holds:
$$
  2\lambda_h\sum_{n\in\Z}\pi_n(\rho_n^\alpha-1)
	\le \sum_{n\in\Z}\frac{\sqrt{\pi_{n+1}\pi_n}}{h^2}
	\big(\rho^{\alpha-1}_{n+1}-\rho^{\alpha-1}_n\big)\big(\rho_{n+1}-\rho_n\big).
$$
\end{theorem}}

\begin{remark}\rm
We remark that $\lambda_h\nearrow \lambda$ as $h\to 0$ \cite[Corollary 5.5]{Mie13}.
Thus, the decay rate is asymp\-to\-ti\-cally sharp. 
A modified log-Sobolev inequality with constant $\lambda$
for a finite-difference approximation was proved in \cite{Joh15}
for $\lambda$-log-concave potentials
by translating the Bakry-Emery condition to the discrete case.
\qed
\end{remark}

\begin{proof}
Note that $a(n)$ and $b(n)$ satisfy the detailed-balance condition \eqref{41.db}.
The proof is a consequence of Theorem \ref{thm.bdp} and the results of 
Mielke \cite[Section 5]{Mie13}. In particular, he has shown that
$(1-\lambda_h)\pi_n \ge \sqrt{\pi_{n-1}\pi_{n+1}}$. Consequently,
\begin{align*}
  a(n)-a(n+1) &= \sqrt{\frac{\pi_{n+1}}{\pi_n}} - \sqrt{\frac{\pi_{n+2}}{\pi_{n+1}}}
	\ge \lambda_h \sqrt{\frac{\pi_{n+1}}{\pi_{n}}}, \\
	b(n+1)-b(n) &= \sqrt{\frac{\pi_{n}}{\pi_{n+1}}}
	- \sqrt{\frac{\pi_{n-1}}{\pi_{n}}} \ge \lambda_h\sqrt{\frac{\pi_{n}}{\pi_{n+1}}}.
\end{align*}
Using Lemma \ref{lem.theta3} and the relation between the arithmetic and
geometric mean, it follows that
\begin{align*}
  a(n) & -a(n+1)+b(n+1)-b(n) + \Theta\big(a(n)-a(n+1),b(n+1)-b(n)\big) \\
	&\ge \alpha\big(a(n)-a(n+1)+b(n+1)-b(n)\big) \\
	&\ge 2\alpha\sqrt{(a(n)-a(n+1))(b(n+1)-b(n))} \ge 2\alpha\lambda_h.
\end{align*}
Applying Theorem \ref{thm.bdp} concludes the proof.
\end{proof}


\begin{appendix}
\section{Properties of the mean function}\label{sec.app}

We show some properties for 
\begin{equation}\label{a.theta}
  \theta(s,t) = \frac{s-t}{\phi'(s)-\phi'(t)}, \quad 0<s,t,<\infty,\ s\neq t,
\end{equation}
with $\theta(s,s)=1/\phi''(s)$. This function
is symmetric and, if $\phi$ is convex, positive. For the following 
lemma, we introduce for $0<s$, $t<\infty$,
$$
  Y(s,t) = (\phi')^{-1}((1-m)\phi'(s) + m\phi'(t)), \quad 0\le m\le 1.
$$
We set $Y_1=\pa Y/\pa s$, $Y_2=\pa Y/\pa t$, $Y_{11}=\pa^2Y/\pa s^2$, etc.

\blue{
\begin{lemma}[Concavity of $\theta$]\label{lem.phi1}
Let $\phi\in C^3((0,\infty);(0,\infty))$ be convex such that $\phi(1)=0$,
and $1/\phi''$ is concave on $(0,\infty)$.
If $\phi^{(3)}(s)\le 0$ for $s>0$, the function $\theta$, defined in \eqref{a.theta},
is nondecreasing in $s$ and in $t$. Furthermore, if additionally
\begin{equation}\label{ym}
  Y_{11} \le 0, \quad Y_{22} \le 0, \quad
	Y_{11}Y_{22} \ge Y_{12}^2 \quad\mbox{in }(0,\infty)^2, \ m\in(0,1),
\end{equation}
then $\theta$ is concave. In this situation, it holds that for all 
$u$, $v$, $s$, $t>0$,
\begin{equation}\label{uvst}
  \theta(u,v)-\theta(s,t) \le \pa_1\theta(s,t)(u-s) + \pa_2\theta(s,t)(v-t).
\end{equation}
\end{lemma}
}

\begin{proof}
The function $\theta$ is nondecreasing in $s$ if and only if $\pa_1\theta(s,t)\ge 0$.
Since 
$$
  \pa_1\theta(s,t) = \frac{\phi'(s)-\phi'(t)-(s-t)\phi''(s)}{(\phi'(s)-\phi'(t))^2},
$$
it is sufficient to prove the nonnegativity of
$G(s,t) = \phi'(s)-\phi'(t)-(s-t)\phi''(s)$. By assumption, the derivative
$\pa_1 G(s,t) = -(s-t)\phi^{(3)}(s)$ is nonpositive for $s\in(0,t)$ and
nonnegative otherwise. Then $G(s,t)\ge G(t,t)=0$, and the conclusion follows.
The monotonicity in the second variable is shown analogously.

For the proof of the concavity of $\theta$, we observe that
$$
  \theta(s,t) =\int_0^1((\phi')^{-1})'\big((1-m)\phi'(s)+ m\phi'(t)\big)dm.
$$
Thus, the concavity of $\theta$ is equivalent to that one of
$$
  F(s,t) = ((\phi')^{-1})'\big((1-m)\phi'(s)+ m\phi'(t)\big)
	= \frac{1}{\phi''(Y(s,t))}
$$
for any $m\in(0,1)$. Let $0<s$, $t<\infty$ and $0<m<1$. 
We claim that if $\phi^{(3)}\le 0$ and
\eqref{ym} holds, then $F$ is concave. For this, it is sufficient to prove that
$F_{11}=\pa^2 F/\pa s^2\le 0$, $F_{22}=\pa^2 F/\pa t^2\le 0$, and
the determinant of the Hessian of $F$ is nonnegative. 
Because of \eqref{ym} and $\phi''(Y)\ge 0$, $\phi^{(3)}(Y)\le 0$, and
$(1/\phi'')''(Y)\le 0$, we obtain
\begin{align*}
  F_{11} &= -\frac{\phi^{(4)}(Y)}{\phi''(Y)^2}Y_1^2
	+ 2\frac{\phi^{(3)}(Y)^2}{\phi''(Y)^3}Y_1^2
	- \frac{\phi^{(3)}(Y)}{\phi''(Y)^2}Y_{11} 
	= \bigg(\frac{1}{\phi''}\bigg)''(Y)Y_1^2
	- \frac{\phi^{(3)}(Y)}{\phi''(Y)^2}Y_{11} \le 0, \\
	F_{22} &= \bigg(\frac{1}{\phi''}\bigg)''(Y)Y_2^2
	- \frac{\phi^{(3)}(Y)}{\phi''(Y)^2}Y_{22} \le 0.
\end{align*}
Then, using the assumptions and
\begin{align*}
  F_{12} &= F_{21} = -\frac{\phi^{(4)}(Y)}{\phi''(Y)^2}Y_1Y_2
	+ 2\frac{\phi^{(3)}(Y)^2}{\phi''(Y)^3}Y_1Y_2
	- \frac{\phi^{(3)}(Y)}{\phi''(Y)^2}Y_{12} \\
	&= \bigg(\frac{1}{\phi''}\bigg)''(Y)Y_1Y_2	
	- \frac{\phi^{(3)}(Y)}{\phi''(Y)^2}Y_{12}, \\
  Y_1 &= (1-m)\frac{\phi''(s)}{\phi''(Y)} \ge 0, \quad
	Y_2 = m\frac{\phi''(t)}{\phi''(Y)} \ge 0, \\
	Y_{12} &= -m(1-m)\frac{\phi''(s)\phi''(t)\phi^{(3)}(Y)}{\phi''(Y)^3}\ge 0,
\end{align*}
it follows that
\begin{align*}
  F_{11}F_{22}-F_{12}^2
	&= \bigg(\frac{\phi^{(3)}(Y)}{\phi''(Y)^2}\bigg)^2
	\big(Y_{11}Y_{22}-Y_{12}^2\big) \\
	&\phantom{xx}{}+ \frac{\phi^{(3)}(Y)}{\phi''(Y)^2}\bigg(\frac{1}{\phi''}\bigg)''(Y)
	\big(2Y_1Y_2Y_{12} - Y_1^2Y_{22} - Y_2^2Y_{11}\big) \ge 0.
\end{align*}
Finally, inequality \eqref{uvst} follows after Taylor expansion and taking into
account the concavity of $\theta$.
\end{proof}

We claim that the assumptions of Lemma \ref{lem.phi1} are satisfied for
the power mean
$$
  \theta_\alpha(s,t) = \frac{\alpha-1}{\alpha}\frac{s-t}{s^{\alpha-1}-t^{\alpha-1}},
	\quad 1<\alpha<2.
$$

\begin{lemma}\label{lem.theta1}
Let $1<\alpha<2$.
The function $\theta_\alpha$ is $C^\infty$, symmetric, positive, increasing
and concave on $(0,\infty)^2$. Furthermore, $\theta_\alpha$ and its first partial 
derivatives are positive homogenous, i.e.,
$\theta_\alpha(\lambda s,\lambda t)=\lambda^{2-\alpha}\theta_\alpha(s,t)$, 
$\pa_1\theta_\alpha(\lambda s,\lambda t)=\lambda^{1-\alpha}\pa_1\theta_\alpha(s,t)$, 
and $\pa_2\theta_\alpha(\lambda s,\lambda t)
=\lambda^{1-\alpha}\pa_2\theta_\alpha(s,t)$ for all $s$, $t>0$ and $\lambda>0$.
\end{lemma}

\begin{proof}
The regularity, symmetry, and positivity of $\theta_\alpha$ follow from elementary
computations. The monotonicity follows from 
$\phi_\alpha^{(3)}(s)=\alpha(\alpha-2)s^{\alpha-3}<0$ for $s>0$. 
To show that $\theta_\alpha$
is concave, we verify the conditions of Lemma \ref{lem.phi1}. We compute
\begin{align*}
  Y(s,t) &= \big((1-m)s^{\alpha-1} + mt^{\alpha-1}\big)^{1/(\alpha-1)}, \\
	Y_{11}(s,t) &= -m(1-m)(2-\alpha)(st)^{\alpha-3}Y(s,t)^{3-2\alpha}t^2, \\
	Y_{22}(s,t) &= -m(1-m)(2-\alpha)(st)^{\alpha-3}Y(s,t)^{3-2\alpha}s^2, \\
	Y_{12}(s,t) &= m(1-m)(2-\alpha)(st)^{\alpha-3}Y(s,t)^{3-2\alpha}st,
\end{align*}
and it follows that $Y_{11}\le 0$, $Y_{22}\le 0$, and $Y_{11}Y_{22}-Y_{12}^2=0$.
\end{proof}

We prove more properties of $\theta_\alpha$, needed in 
Sections \ref{sec.zrp}-\ref{sec.rtm}.

\begin{lemma}[Properties of $\theta_\alpha$]\label{lem.theta2}
Let $1<\alpha<2$.
The function $\theta_\alpha$ satisfies for all $r$, $s$, $t>0$ and $\lambda_1$,
$\lambda_2>0$, \\
{\rm (i)} $s\pa_1\theta_\alpha(s,t) + t\pa_2\theta_\alpha(s,t) 
= (2-\alpha)\theta_\alpha(s,t)$; \\
{\rm (ii)} $2^{\alpha-1}r(\pa_1\theta_\alpha(s,t)+\pa_2\theta_\alpha(s,t)) 
- (\theta_\alpha(r,s) + \theta_\alpha(r,t)) \ge -2^{\alpha-1}\theta_\alpha(s,t)$; \\
{\rm (iii)} $\lambda_1\pa_1\theta_\alpha(s,t)(s-t) 
- \lambda_2\pa_2\theta_\alpha(s,t)(s-t)
\le (2-\alpha)|\lambda_1-\lambda_2|\theta_\alpha(s,t)$.
\end{lemma}

\begin{proof}
Identity (i) can be obtained by an elementary computation.
The proof of (ii) is similar to the proof of Lemma A.2 in \cite{FaMa15}.
Indeed, setting $u=s/r$ and $v=t/r$ and using the homogeneity properties of
$\theta_\alpha$ and its first partial derivatives, inequality (ii) is equivalent to
$$
  2^{\alpha-1}\big(\pa_1\theta_\alpha(u,v) + \pa_2\theta_\alpha(u,v)\big)
	- \big(\theta_\alpha(1,u)+\theta_\alpha(1,v)\big) 
	\ge -2^{\alpha-1}\theta_\alpha(u,v).
$$
This inequality follows from the concavity and the $(2-\alpha)$-homogeneity 
property of $\theta_\alpha$ and from (i):
\begin{align*}
  \theta_\alpha(1,u)+\theta_\alpha(1,v) 
	&\le 2\theta_\alpha\left(\frac{u+1}{2},\frac{v+1}{2}\right)
  = 2^{\alpha-1}\theta_\alpha(u+1,v+1) \\
	&\le 2^{\alpha-1}\big(\theta_\alpha(u,v) + \pa_1\theta_\alpha(u,v) 
	+ \pa_2\theta_\alpha(u,v)\big).
\end{align*}

\blue{Finally, by property (i),
\begin{align*}
  \lambda_1 & \pa_1\theta_\alpha(s,t)(s-t) - \lambda_2\pa_2\theta_\alpha(s,t)(s-t) \\
	&\le \max\{\lambda_1,\lambda_2\}\big(s\pa_1\theta_\alpha(s,t) 
	+ t\pa_2\theta_\alpha(s,t)\big) 
	- \min\{\lambda_1,\lambda_2\}\big(t\pa_1\theta_\alpha(s,t) 
	+ s\pa_2\theta_\alpha(s,t)\big) \\
	&= \max\{\lambda_1,\lambda_2\}(2-\alpha)\theta_\alpha(s,t)
	- \min\{\lambda_1,\lambda_2\}\big(t\pa_1\theta_\alpha(s,t) 
	+ s\pa_2\theta_\alpha(s,t)\big).
\end{align*}
Choosing $u=t$ and $v=s$ in \eqref{uvst} gives
$\pa_1\theta_\alpha(s,t)(s-t) + \pa_2\theta_\alpha(s,t)(t-s) \le 0$,
and combining this inequality with property (i) yields
\begin{align*}
  -\big(t & \pa_1\theta_\alpha(s,t) 
	+ s\pa_2\theta_\alpha(s,t)\big) 
	= \pa_1\theta_\alpha(s,t)(s-t) + \pa_2\theta_\alpha(s,t)(t-s) \\
	&{}- \big(s\pa_1\theta_\alpha(s,t) + t\pa_2\theta_\alpha(s,t)\big)
	\le -(2-\alpha)\theta_\alpha(s,t),
\end{align*}
such that
\begin{align*}
  \lambda_1 & \pa_1\theta_\alpha(s,t)(s-t) - \lambda_2\pa_2\theta_\alpha(s,t)(s-t) \\
	&\le \big(\max\{\lambda_1,\lambda_2\} - \min\{\lambda_1,\lambda_2\}\big)
	(2-\alpha)\theta_\alpha(s,t) \\
	&= |\lambda_1-\lambda_2|(2-\alpha)\theta_\alpha(s,t).
\end{align*}}
This concludes the proof.
\end{proof}

\blue{
\begin{lemma}\label{lem.theta3}
Let $\phi_\alpha(s)=(s^\alpha-s)/(\alpha-1)-s+1$ and $1<\alpha<2$.
It holds for all $A$, $B\ge 0$,
$$
  \Theta(A,B) := \inf_{s,t>0}\theta_\alpha(s,t)(A\phi_\alpha''(s)+B\phi_\alpha''(t)) 
	\ge (\alpha-1)(A+B).
$$
\end{lemma}}

\begin{proof}
Since
$$
  \theta_\alpha(s,t) =\frac{\alpha-1}{\alpha}\frac{s-t}{s^{\alpha-1}-t^{\alpha-1}}
	= \frac{1}{\alpha}\int_0^1\big((1-m)s^{\alpha-1} 
	+ mt^{\alpha-1}\big)^{(2-\alpha)/(\alpha-1)}dm,
$$
it follows that
\begin{align*}
  \theta_\alpha(s,t)(A\phi_\alpha''(s)+B\phi_\alpha''(t)) 
	&= A\int_0^1\bigg((1-m) + m\left(\frac{t}{s}\right)^{\alpha-1}
	\bigg)^{(2-\alpha)/(\alpha-1)} dm \\
	&\phantom{xx}{}+ B\int_0^1\bigg((1-m)\left(\frac{s}{t}
	\right)^{\alpha-1} + m\bigg)^{(2-\alpha)/(\alpha-1)} dm \\
	&\ge A\int_0^1 (1-m)^{(2-\alpha)/(\alpha-1)} dm
	+ B\int_0^1 m^{(2-\alpha)/(\alpha-1)} dm \\
	&= (\alpha-1)(A+B),
\end{align*}
which finishes the proof.
\end{proof}

\end{appendix}


\end{document}